\documentclass[thmsa]{article}
\usepackage{amsmath}
\usepackage{amssymb}
\usepackage{amsthm,amscd,latexsym}
\usepackage{color}

\newtheorem{theorem}{Theorem}

\newtheorem{corollary}{Corollary}

\newtheorem{lemma}{Lemma}
\newtheorem{proposition}{Proposition}
\theoremstyle{definition}
\newtheorem{remark}{Remark}

\begin{document}

\title{Asymptotic behaviour of first passage time distributions for
subordinators}
\author{R. A. Doney\thanks{%
Department of Mathematics, University of Manchester, Manchester M13 9PL,
United Kingdom, email: ron.doney@manchester.ac.uk}~\ and V. Rivero\thanks{%
Centro de Investigacion en Matematicas A.C., Calle Jalisco s/n, 36240
Guanajuato, Me\-xi\-co, email: rivero@cimat.mx} \thanks{%
Research funded by the CONACYT Project \textit{Teor{\'\i }a y aplicaciones
de procesos de L{\'e}vy}} }
\maketitle

\abstract{In this paper we establish local estimates for the first passage
time of a subordinator under the assumption that it belongs to the Feller
class, either at zero or infinity, having as a particular case the
subordinators which are in the domain of attraction of a stable
distribution, either at zero or infinity.  To derive these results we first obtain uniform local estimates for the one dimensional distribution
of such a subordinator, which sharpen those obtained by Jain and
Pruitt~\cite{J+P}.  In the particular case of a subordinator in the domain
of attraction of a stable distribution our results are the analogue of the
results obtained by the authors in \cite{doney-rivero} for non-monotone
L\'evy processes.  For subordinators an approach different to that in 
\cite{doney-rivero} is necessary because the excursion techniques are not
available and also because typically in the non-monotone case the tail
distribution of the first passage time has polynomial decrease, while in the
subordinator case it is exponential.}

~\newline



\section{Introduction and main results}

Let $X$ be a subordinator, a stochastic process with non-decreasing c\`{a}dl%
\`{a}g paths with independent and stationary increments, with Laplace
exponent $\psi ,$ 
\begin{equation*}
-\frac{1}{t}\log \left( \mathbb{E}(\exp \{-\lambda X_{t}\})\right) =:\psi
(\lambda )=b\lambda +\int_{(0,\infty )}(1-e^{-\lambda x})\Pi (dx),\qquad
\lambda \geq 0,
\end{equation*}%
where $b$ denotes the drift and $\Pi $ the L\'{e}vy measure of $X.$ We will
write $\psi _{\ast }$ for the exponent of $\{X_{t}-bt,t\geq 0\},$ so that 
\begin{equation*}
\psi _{\ast }(\lambda ):=\psi (\lambda )-b\lambda ,\qquad \lambda \geq 0.
\end{equation*}

We are interested in determining the \textbf{local }asymptotic behaviour of
the distribution of $T_{x}=\inf \{t>0:X_{t}>x\}.$ More precisely, we would
like to establish estimates for the density function $h_{x}(t),$ (if it
exists: it does if $b=0),$ or more generally of 
\begin{equation*}
\mathbb{P}(T_{x}\in (t,t+\Delta ]),\quad
\end{equation*}%
uniformly for $\Delta $ in bounded sets and uniformly for $x$ in certain
regions, both as $t\rightarrow \infty $ or as $t\rightarrow 0$. Knowing the
first passage time distribution of a subordinator is of central importance
because of its applications in stochastic modeling and theoretical
probability, see for instance \cite{bertoinsubordinators}.

This paper is a continuation of recent research in \cite{doney-rivero},
where the same problem, in the $t\rightarrow \infty $ case, has been solved
for L\'{e}vy processes, excluding subordinators, that are in the domain of
attraction of a stable law without centering. The reasons for excluding
subordinators from that research were that the techniques used there rely
heavily on excursion theory for the reflected process, which in this case
does not make sense, and that in the subordinators case the rate of decrease
of the tail distribution of the first passage time is typically exponential,
while for other L\'{e}vy processes it is polynomial.

As can be seen in the paper~\cite{doney-rivero}, and in the present case,
the distribution of the first passage time has different behaviour according
to whether the process first crosses the barrier by a jump or continuously,
that is by \textit{creeping}. So, our results will describe the
contributions of these events to the first passage time distribution
separately. Of course if a subordinator has zero drift, it cannot creep, and
moreover the distribution of $T_{x}$ is absolutely continuous, so our
results become somewhat simpler in that case.

Our main results are quite complicated to state, in part because as well as
treating asymptotically stable subordinators, we consider the more general
case of \ stochastically compact subordinators. But we can illustrate our
results by looking at the case when $X$ is exactly stable with $\alpha =1/2.$
That means that $b=0$, $\Pi (dx)=cx^{-\frac{3}{2}}dx,$ and $\psi (\lambda )=%
\sqrt{\lambda }.$ In this case we know that $X_{t}$ is absolutely continuous
with density%
\begin{equation*}
f_{t}(x)=\frac{t}{\sqrt{2\pi x^{3}}}e^{-\frac{t^{2}}{2x}},
\end{equation*}%
and straightforward calculations give 
\begin{equation}
h_{x}(t)=\int_{0}^{x}\mathbb{P}(X_{t}\in dy)\overline{\Pi }(x-y)=\frac{1}{%
\sqrt{\pi x}}e^{-\frac{t^{2}}{2x}},  \label{h}
\end{equation}%
\begin{equation}
\mathbb{P}(T_{x}\in (t,t+\Delta ])\backsim \frac{t}{2x}(1-e^{-\frac{2x\Delta 
}{t}})h_{x}(t)\text{ uniformly as }\frac{t^{2}}{x}\rightarrow \infty .
\label{h2}
\end{equation}%
Even in this simple example there is something surprising: if $x\Delta
/t\nrightarrow 0$ the RHS of (\ref{h2}) is not asymptotic to $\Delta
h_{x}(t).$

The condition $\frac{t^{2}}{x}\rightarrow \infty $ is equivalent to $\mathbb{%
P}(T_{x}>t)=\mathbb{P}(X_{t}\leq x)\rightarrow 0,$ and it was shown in \cite%
{J+P} that the corresponding condition for a general subordinator is that $%
tH(\rho )\rightarrow \infty ,$ where $H(u)=\psi (u)-u\psi ^{\prime }(u),$
and $\rho $ is the unique solution of $\psi ^{\prime }(\rho )=x/t,$ so that
in our example $\rho =(\frac{t}{2x})^{2}$ and $H(\rho )=t/x$. Here $\rho $
is the parameter of the exponential change of measure which generates a
subordinator having mean $x/t,$ and we will mainly be interested in the case
that $x/t\rightarrow 0,$ when $\rho \rightarrow \infty ,$ or $x/t\rightarrow
\infty ,$ when $\rho \rightarrow 0.$ Note that the first case can arise with 
\textbf{either} $t\rightarrow \infty $ or $t\rightarrow 0,$ but the second
case can only arise with $t\rightarrow \infty .$

The variance of the transformed subordinator is $\sigma ^{2}(\rho ),$ where $%
\sigma ^{2}(u)=\psi ^{\prime \prime }(u),$ which is $2x^{3}/t^{3}$ in the
example. We therefore see that the above are special cases of 
\begin{equation}
f_{t}(x)\backsim \sqrt{\frac{1}{2\pi t\sigma ^{2}(\rho )}}e^{-tH(\rho )},
\label{s1}
\end{equation}%
\begin{equation}
h_{x}(t)\backsim \frac{\psi (\rho )}{\rho }f_{t}(x),  \label{s2}
\end{equation}%
and%
\begin{equation}
\mathbb{P}(T_{x}\in (t,t+\Delta ])\backsim \frac{1}{\rho }(1-e^{-\Delta \psi
(\rho )})f_{t}(x).  \label{s3}
\end{equation}%
Suppose now that $X$ is any driftless subordinator having $\overline{\Pi }%
(x) $ regularly varying as $x\rightarrow 0$ with index $-\alpha $ satisfying 
$0\,<\alpha <\,1,$ so that $X_{t}$ is asymptotically stable as $t\rightarrow
0.$ Then $X$ is absolutely continuous, and our first main result shows that (%
\ref{s1}) is valid as $x/t\rightarrow 0,$ both as $t\rightarrow 0$ and as $%
t\rightarrow \infty .$ (This significantly improves a result about the
behaviour of $\mathbb{P}(X_{t}\leq x)$ in \cite{J+P}.) Using the
representation for $h_{x}(t)$ in (\ref{h}) it is then straightforward to
verify that (\ref{s2}) holds. In principle, this must imply (\ref{s3}), but
it turns out to be easier to deduce (\ref{s3}) from (\ref{s1}), using the
formula 
\begin{equation*}
\begin{split}
& \mathbb{P}(T_{x}\in (t,t+\Delta ],X_{T_{x}}>x) \\
& =\int_{[0,x)}\mathbb{P}(X_{t}\in x-dy)\int_{[0,y]}U_{\Delta }(dz)\overline{%
\Pi }(y-z).
\end{split}%
\end{equation*}%
(Here $U_{\Delta }(dz)=\int_{0}^{\Delta }\mathbb{P}(X_{s}\in dz)ds,$ see
Lemma\ \ref{CJF}). Moreover if in this situation the subordinator has drift $%
b>0,$ it is a consequence of the results in \cite{griffinmaller} that%
\begin{equation}
\mathbb{P}(T_{x}\in dt,X(T_{x})=x)=bf_{t}(x)dt\text{ for }x>bt,  \label{cr}
\end{equation}%
see Lemma~\ref{CJF}. Since $X$ jumps over $x$ at time $t$ if $X-bt$ jumps
over $x-bt$ at time $t,$ we see that (\ref{s2}) holds for $x>bt$ with $\psi
(\rho )$ replaced by $\psi (\rho )-b\rho ,$ and it follows that 
\begin{equation}
\mathbb{P}(X\text{ creeps over }x|T_{x}=t)\backsim \frac{b\rho }{\psi (\rho )%
}\rightarrow \left\{ 
\begin{array}{ccc}
1, & \text{if} & \rho \rightarrow \infty , \\ 
\frac{b}{\psi ^{\prime }(0+)}, & \text{if} & \rho \rightarrow 0.%
\end{array}%
\right.  \label{conditional}
\end{equation}%
The case $t\rightarrow \infty $ and $x/t\rightarrow \infty $ is more
difficult: here the natural assumption is that $X_{t}$ is asymptotically
stable as $t\rightarrow \infty ,$ or equivalently that $\overline{\Pi }(x)$
is regularly varying as $x\rightarrow \infty $ with index $-\alpha $
satisfying $0\,<\alpha <\,1$ and $b\geq 0.$ Here $X$ is \textbf{not}
automatically absolutely continuous, and although we can get an estimate for 
$\mathbb{P}(X_{t}\in (x,x+\Delta ])$ analogous to (\ref{s1}), this is only
valid for $\Delta $ in compact\ sub-intervals of $(0,\infty ).$ The possible
singularity of $\overline{\Pi }(x-y)$ at $y=x$ then seems to make a direct
calculation based on (\ref{h}) impossible. Instead we exploit the fact that
the RHS of (\ref{h}) for fixed $t$ is a convolution and use the inversion
theorem for characteristic functions to establish (\ref{s2}) by an indirect
method. Unfortunately this requires us to make an additional assumption
which involves the behaviour of $\Pi $ near zero, (see (H) below), but
crucially this assumption does not necessitate $X$ being absolutely
continuous. When dealing with the creeping component we therefore cannot
rely on (\ref{cr}), and instead we look for an estimate of 
\begin{equation*}
h_{x}^{C}(t,\Delta ):=\mathbb{P}(T_{x}\in (t,t+\Delta ],X_{T_{x}}=x),
\end{equation*}%
for which we use the formula%
\begin{equation*}
\mathbb{P}(T_{x}\in (t,t+\Delta ],X_{T_{x}}=x)=b\int_{[0,x)}\mathbb{P}%
(X_{t}\in dy)u_{\Delta }(x-y).
\end{equation*}%
(Here $u_{\Delta }(z)dz=U_{\Delta }(dz)$; see Lemma\ \ref{CJF}). In the case
that $X$ is stable-$1/2$ plus a drift $b>0$ one can check directly that the
result is%
\begin{eqnarray}
h_{x}^{C}(t,\Delta ) &\backsim &\frac{b\rho }{\psi (\rho )}\mathbb{P}%
(T_{x}\in (t,t+\Delta ])\text{ and}  \label{P1} \\
h_{x}^{J}(t,\Delta ) &\backsim &\left( 1-\frac{b\rho }{\psi (\rho )}\right) 
\mathbb{P}(T_{x}\in (t,t+\Delta ]),  \label{P2}
\end{eqnarray}%
where the asymptotic behaviour of $\mathbb{P}(T_{x}\in (t,t+\Delta ])$ is
given by the RHS of (\ref{h2}) evaluated with $x$ replaced by $x-bt.$ As we
will see, (\ref{P1}) and (\ref{P2}) hold in the general case, as does (\ref%
{conditional}). We also have analogous results for the regime where $\rho $
is bounded away from zero and infinity and $t\rightarrow \infty $, where we
make no assumptions about $X$ other than (H) and that it is a strongly
non-lattice subordinator.

As we already mentioned, in the present work we allow a more general
behaviour than that of being in the domain of attraction of a stable law,
namely for most of our results we only require $X$ to be in the \textit{%
Feller class}, said otherwise to be \textit{stochastically compact}, either
at infinity or at zero depending on whether $x/t$ tends to $b$ from above,
or to $\mathbb{E}(X_{1})$ from below, or is bounded away from $b$ and $%
\mathbb{E}(X_{1}).$ (This class includes subordinators for which $\overline{%
\Pi }(\cdot )$ is O-regularly varying at zero and infinity, see e.g. \cite%
{BGT}.) A further difference from our work in \cite{doney-rivero} is that
the results here obtained apply equally to subordinators which are
stochastically compact with or without centering, while in \cite%
{doney-rivero} the assumption that the L\'{e}vy process is in the domain of
attraction of a stable law \textbf{without} centering is in force.

In order to provide precise definitions of these notions we start by
introducing some notation. We will write 
\begin{equation}
H(u)=\psi (u)-u\psi ^{\prime }(u),\qquad \sigma ^{2}(u)=\int_{0}^{\infty
}y^{2}e^{-u{y}}\Pi (dy),\qquad u\geq 0,
\end{equation}%
and for $x>0,$ 
\begin{align}
\overline{\Pi }(x)=\Pi (x,\infty )& ,\quad K_{\Pi }(x)=x^{-2}\int_{y\in
(0,x)}y^{2}\Pi (dy), \\
& Q_{\Pi }(x)=\overline{\Pi }(x)+K_{\Pi }(x).
\end{align}%
An elementary verification shows that 
\begin{equation}
Q_{\Pi }(z)=2z^{-2}\int_{0}^{z}y\overline{\Pi }(y)dy,\qquad z>0,  \label{qp}
\end{equation}%
and that $Q_{\Pi }$ is a non-increasing function. We define $\rho :\mathbb{R}%
^{+}\rightarrow \mathbb{R}^{+}$ via the relation 
\begin{equation*}
\psi ^{\prime }(\rho (s))=s,\qquad 0\leq b=\psi ^{\prime }(\infty )<s<\psi
^{\prime }(0+)=\mathbb{E}(X_{1})=:\mu \leq \infty .
\end{equation*}%
From this relation it is easily seen that $\rho (\cdot )$ is a
non-increasing function. For notational convenience for $b<x_{t}<\mu ,$ we
will write $\rho_{t}:=\rho \left( x_{t}\right) .$ Note that $\rho
_{t}\downarrow 0$ when $x_{t}\uparrow \mu $ and $\rho_{t}\uparrow \infty $
when $x_{t}\downarrow b.$

We will say that $X$ is in a \textit{Feller class} or is \textit{%
stochastically compact} at infinity, respectively at $0$, if

\begin{itemize}
\item[{\textbf{[SC]}}] $\displaystyle\limsup\frac{\overline{\Pi}(y)}{%
K_{\Pi}(y)}<\infty$ as $y\to\infty,$ respectively as $y\to 0{+}.$
\end{itemize}

It is known that this condition is equivalent to

\begin{itemize}
\item[{\textbf{[SC']}}] $\exists \alpha\in(0,2]$ and $c\geq 1$ such that $%
\displaystyle\limsup\frac{\int^{\lambda z}_{0}y\overline{\Pi}(y)dy}{%
\int^{z}_{0}y\overline{\Pi}(y)dy}\leq c\lambda^{2-\alpha}$ for $\lambda>1,$
as $z\to\infty,$ respectively as $z\to 0+;$
\end{itemize}

see \cite{Maller+MasonTAMS} for a proof of this equivalence and background
on the study of the Feller class for general L\'{e}vy processes. In this
case we will say that the condition $SC_{\infty }$, respectively $SC_{0},$
holds. We next quote some facts from the work by Maller and Mason in \cite%
{maller+masonat0} and \cite{Maller+MasonTAMS}. In the case where $X$ is
stochastically compact at infinity (respectively at zero), Maller and Mason
proved that there exist functions $\mathrm{c}:[0,\infty )\rightarrow
(0,\infty )$ and $\mathbf{b}:[0,\infty )\rightarrow \lbrack 0,\infty )$ such
that for any sequence $(t_{k},k\geq 0)$ tending towards infinity
(respectively, towards 0) there is a subsequence $(t_{k}^{\prime },k\geq 0)$
such that 
\begin{equation}  \label{seqlimit}
\frac{X_{t_{k}^{\prime }}-\mathbf{b}(t_{k}^{\prime })}{\mathrm{c}%
(t_{k}^{\prime })}\xrightarrow[k\to\infty]{\text{Law}}Y^{\prime },
\end{equation}%
where $Y^{\prime }$ is a real valued non-degenerate random variable, whose
law may depend on the subsequence taken. A standard representation of the
functions $\mathrm{c}$ and $\mathbf{b}$ are 
\begin{equation}
tQ_{\Pi }(\mathrm{c}(t))=1,\qquad \mathbf{b}(t)=t\left(b+\int_{0}^{\mathrm{c}%
(t)}y\Pi (dy)\right),\qquad t>0.
\end{equation}%
If in addition to the condition $SC_{\infty}$ (respectively $SC_{0}$) the
condition 
\begin{equation}  \label{SC00}
\limsup_{y\rightarrow \infty (y\to 0) }\frac{y(b+\int_{0}^{y}z\Pi (dz))}{%
\int_{0}^{y}z^{2}\Pi (dz)}<\infty ,
\end{equation}%
holds, then the above defined functions satisfy 
\begin{equation}
\limsup_{t\rightarrow \infty (t\rightarrow 0) }\frac{\mathbf{b}(t)}{c(t)}%
<\infty ,
\end{equation}
so that the normalizing function $\mathbf{b}$ is not needed and hence can be
assumed to be $0.$ In this case it is said that the process $X$ is
stochastically compact at zero (respectively at infinity) without centering.
In all other cases, 
\begin{equation}
\limsup_{t\rightarrow \infty (t\rightarrow 0) }\frac{\mathbf{b}(t)}{c(t)}%
=\infty .
\end{equation}

Throughout the paper we will work in one of the following frameworks on $\Pi
,$ $t$ and $x$: always $b<x_{t}:=x/t<\mu $ and

\begin{itemize}
\item[($SC_{0}$-I)] the L\'evy measure $\Pi$ satisfies the condition $%
SC_{0}, $ $t\to\infty,$ $x_{t}\to b;$

\item[($SC_{0}$-II)] the L\'evy measure $\Pi$ satisfies the condition $%
SC_{0},$ $t\to0,$ $x_{t}\to b,$ and 
\begin{equation*}
\frac{x-\mathbf{b}(t)}{c(t)}\xrightarrow[t\to 0]{}-\infty
\end{equation*}
if (\ref{SC00}) fails; or 
\begin{equation*}
\frac{x-bt}{c(t)}\xrightarrow[t\to 0]{}0
\end{equation*}
if (\ref{SC00}) holds.

\item[($SC_{\infty }$)] the L\'{e}vy measure $\Pi $ satisfies the condition $%
SC_{\infty },$ $t\rightarrow \infty ,$ $x_{t}\rightarrow \mu $ and 
\begin{equation*}
\frac{x-\mathbf{b}(t)}{c(t)}\xrightarrow[t\to \infty]{}-\infty
\end{equation*}%
if (\ref{SC00}) fails; or 
\begin{equation*}
\frac{x-bt}{c(t)}\xrightarrow[t\to \infty]{}0
\end{equation*}%
if (\ref{SC00}) holds.

\item[$(G)$] $t\rightarrow \infty $ and $b<\liminf_{t\rightarrow \infty }%
\frac{x}{t}\leq \limsup_{t\rightarrow \infty }\frac{x}{t}<\mu,$ and $X$ is
strongly non-lattice.
\end{itemize}

Furthermore, to work in the cases $(SC_{\infty})$ and $(G)$ we will need a
supplementary hypothesis

\begin{itemize}
\item[(H)] there exists a $t_{0}>0$ such that 
\begin{equation*}
\int_{1}^{\infty }\exp \left\{ -t_{0}\int_{0}^{\infty }(1-\cos (zy))\Pi
(dy)\right\} \frac{1+|\psi_{*} (-iz)|}{z}dz<\infty .
\end{equation*}
\end{itemize}

\begin{remark}
\begin{enumerate}
\item 

\item In the compound Poisson case it is clear that this condition cannot
hold, but in that case $\Pi $ is integrable at zero, so we can use a method
similar to that we use for the SC$_{0}$ cases.

\item The stronger condition%
\begin{equation}
\int_{1}^{\infty }\exp \left\{ -t_{0}\int_{0}^{\infty }(1-\cos (zy))\Pi
(dy)\right\} dz<\infty  \label{AC}
\end{equation}%
would imply that $X_{t}$ has an absolutely continuous distribution for each $%
t>0,$ but it is easy to see that (H) can hold without (\ref{AC}) holding.

\item {\ Using the elementary inequalities $1-\cos(y)\geq \frac{1}{\pi}%
y^{2}, $ for $-1<y<1,$ and $|\sin(y)|\leq 1\wedge |y|$, for $y\in \mathbb{R}%
, $ it can be verified that (H) holds whenever there exists a $t_{0}$ such
that 
\begin{equation*}
\int^{\infty}_{1}\exp\left\{-t_{0}z^{2}\int^{1/z}_{0}b^{2}\Pi(db)\right\} 
\left[\frac{1}{z}+\int^{1/z}_{0}\overline{\Pi}(a)da\right]dz<\infty.
\end{equation*}%
}
\end{enumerate}
\end{remark}

We start by providing some local estimates of the distribution of $X$ in the 
$(SC_{0})$ cases. These are a refinement of the pioneering work by Jain and
Pruitt~\cite{J+P}, which is one of the main sources of this research, where
estimates for $\mathbb{P}(T_{x}>t)=\mathbb{P}(X_{t}\leq x)$ are obtained.
The technique we use is different to that of Jain and Pruitt, though both
techniques involve normal approximations. Throughout this note $\phi:\mathbb{%
R}\to\mathbb{R}^{+},$ will denote the standard normal density.

\begin{theorem}
\label{thlocalXSC0} Suppose that $X$ is a subordinator which has drift $%
b\geq 0$ and L{\'{e}}vy measure $\Pi .$ For $b<x/t<\mu:=\mathbb{E}(X_{1}) ,$
define $x_{t}:=x/t$ and $\rho_{t}:=\rho (x/t)$, that is $\psi ^{\prime
}(\rho _{t})=x/t,$ $H(u)=\psi (u)-u\psi ^{\prime }(u),$ and $\sigma
^{2}(u)=\int_{0}^{\infty }y^{2}e^{-u{y}}\Pi (dy).$

\begin{itemize}
\item[(i)] If $X$ is stochastically compact at $0,$ the unidimensional law
of $X$ admits a density, say $\mathbb{P}(X_{t}\in dy)=f_{t}(y)dy,$ $y\geq 0,$
such that $f_{t}\in C^{\infty }(\mathbb{R})$ for each $t>0.$

\item[(ii)] In the settings ($SC_{0}$-(I-II)) we have the estimate 
\begin{equation}
\sqrt{t}\sigma (\rho_{t})f_{t}(z)=\left( \phi ((z-x)/\sqrt{t}\sigma (\rho
_{t}))+o(1)\right) e^{-tH(\rho_{t})}e^{\rho_{t}(z-x)},  \label{6a}
\end{equation}%
uniformly in $z>0$ and $x.$
\end{itemize}
\end{theorem}

We now turn to the results for the passage time. We are interested in the
probability of $X$ passing above level $x$ in the time interval $(t,t+\Delta
],$ either by a jump or by creeping. The latter is positive only when the
drift $b>0.$ The latter and former probabilities will be denoted by 
\begin{equation*}
h_{x}^{C}(t,\Delta ):=\mathbb{P}(T_{x}\in (t,t+\Delta ],X_{T_{x}}=x),
\end{equation*}%
and 
\begin{equation*}
h_{x}^{J}(t,\Delta ):=\mathbb{P}(T_{x}\in (t,t+\Delta ],X_{T_{x}}>x),\qquad
t>0.
\end{equation*}%
First, in the settings ($SC_{0}$-(I-II)) we have by the forthcoming Lemma~%
\ref{CJF} that $h_{x}^{C}(t,\Delta )=\int_{t}^{t+\Delta }h_{x}^{C}(s)ds$ and 
\begin{equation*}
h_{x}^{C}(t)=bf_{t}(x),\qquad t,x>0,
\end{equation*}%
and by the above Theorem, 
\begin{equation*}
h_{x}^{C}(t)=bf_{t}(x)=\frac{b}{\sqrt{2\pi t}\sigma (\rho _{t})}e^{-tH(\rho
_{t})}(1+o(1))
\end{equation*}%
uniformly in $x.$ Furthermore, in all cases the expression 
\begin{equation*}
h_{x}^{J}(t)=\int_{0}^{x}\mathbb{P}(X_{t}\in dy)\overline{\Pi }(x-y),\qquad
t>0,
\end{equation*}%
is the density function of the first passage time on the event $X_{T_{x}}>x,$
see \cite{doney-rivero} Lemma 1. Then our main result is the following.

\begin{theorem}
\label{thm01} Assume we are in the settings ($SC_{0}$-(I-II)), $(SC_{\infty
})$ or $(G)$ and the hypothesis (H) is satisfied in the last two cases. Put $%
\psi _{\ast }(\lambda )=\psi (\lambda )-\lambda b=\int_{0}^{\infty
}(1-e^{-\lambda y})\Pi (dy).$ We have the following estimates 
\begin{equation}
h_{x}^{J}(t)=\frac{\psi _{\ast }(\rho _{t})}{\sqrt{2\pi t}\rho _{t}\sigma
(\rho _{t})}e^{-tH(\rho _{t})}\left( 1+o(1)\right) ,  \label{eq01}
\end{equation}%
\begin{equation}
h_{x}^{J}(t,\Delta )=\frac{\psi _{\ast }(\rho _{t})\left( 1-e^{-\Delta \psi
(\rho _{t})}\right) }{\psi (\rho _{t})\sqrt{2\pi t}\rho _{t}\sigma (\rho
_{t})}e^{-tH(\rho _{t})}\left( 1+o(1)\right) ,  \label{eq03}
\end{equation}%
\begin{equation}
h_{x}^{C}(t,\Delta )=\frac{b\rho _{t}\left( 1-e^{-\Delta \psi (\rho
_{t})}\right) }{\psi (\rho _{t})\sqrt{2\pi t}\rho _{t}\sigma (\rho _{t})}%
e^{-tH(\rho _{t})}\left( 1+o(1)\right) ,  \label{eq04}
\end{equation}%
uniformly in $x$ and uniformly for $0<$ $\Delta \leq \Delta _{0},$ for any
fixed $\Delta _{0}>0$.
\end{theorem}

In the above estimates if $\Delta$ is bounded away from zero the term $%
(1-e^{-\Delta\psi(\rho_{t})})/\psi(\rho_{t})$ can be replaced by $\Delta$ or 
$1/\psi(\rho_{t}),$ according as $\rho_{t}\to0$ or $\infty.$

The above results can be applied when $X$ is in the domain of attraction of
a stable distribution of order $\alpha ,0<\alpha <1$. In this case we have
the asymptotic equivalence 
\begin{equation*}
\psi _{\ast }(\rho _{t})\backsim c_{\alpha }\overline{\Pi }(1/\rho _{t}),
\end{equation*}%
and this can be used to rephrase the results in .

\begin{corollary}
\label{corlocalT} Define a function $c$ by $t\overline{\Pi }(c(t))=1,$ $t>0. 
$ Then, in particular, the results in Theorem \ref{thm01} hold in the
following cases.

\begin{itemize}
\item[(i)] $t\rightarrow \infty /0$, $\frac{x}{t}\rightarrow b,$ $%
(x-tb)/c(t)\rightarrow 0,$ when $\overline{\Pi }(\cdot )\in RV(-\alpha )$ 
\textbf{at }$0$ with $\alpha \in (0,1).$

\item[(ii)] $t\rightarrow \infty ,$ $x/t\rightarrow \infty ,$ $%
x/c(t)\rightarrow 0,$ when $\overline{\Pi }(\cdot )\in RV(-\alpha )$ \textbf{%
at }$\infty $ with $\alpha \in (0,1),$ and (H) holds.
\end{itemize}
\end{corollary}

\begin{remark}
If (i) holds and $b>0,$ we see that $h_{x}^{J}(t,\Delta )/h_{x}^{C}(t,\Delta
)\rightarrow 0,$ but note in this scenario $X_{t}$ is not in the domain of
attraction of an $\alpha $-stable subordinator without centering as $%
t\rightarrow 0.$ If (ii) holds, and $b>0,$ $X_{t}$ is in the domain of
attraction of an $\alpha $-stable subordinator without centering as $%
t\rightarrow \infty ,$ and this ratio $\rightarrow \infty .$ In this
situation, our forthcoming final result shows that it is possible for
polynomial, rather than exponential decay to occur, but again this ratio $%
\rightarrow \infty $.
\end{remark}

\begin{proposition}
\label{prop:rvcase} Suppose now that $X$ is a strongly non-lattice
subordinator which has drift $b\geq 0$ and $\overline{\Pi }(\cdot )\in
RV(-\alpha )$ at $\infty $ with $\alpha \in (0,1),$ and (H) holds$.$ Define $%
c$ by $t\overline{\Pi }(c(t))=1,$ so that $(X(t\cdot )/c(t))$ converges
weakly to $S,$ a stable subordinator of index $\alpha .$ Let $\tilde{g}%
_{t}(\cdot )$ and $\tilde{h}_{x}(\cdot )$ denote the density functions of $%
S_{t}$ and $T_{x}^{S}:=\inf \{t:S_{t}>x\}$ respectively. Then uniformly for $%
y_{t}:=x/c(t)>0$ 
\begin{equation}
th_{x}^{J}(t)=\tilde{h}_{y_{t}}(1)+o(1)\text{ as }t\rightarrow \infty .
\label{plT1}
\end{equation}%
and, if $b>0,$ uniformly for $y_{t}>0$ and $0<\Delta <\Delta _{0},$ 
\begin{equation}
c(t)h_{x}^{C}(t,\Delta )=b\Delta \left( \tilde{g}_{1}\left( y_{t}\right)
+o(1)\right) \text{ as }t\rightarrow \infty .  \label{plT2}
\end{equation}
\end{proposition}

\section{Preliminaries}

For sake of conciseness we gather some useful formulas in the following
Lemma.

\begin{lemma}
\label{CJF} Let $X$ be a subordinator with Laplace exponent $\psi,$ drift $%
b\geq 0,$ and L\'evy measure $\Pi.$ We have the following facts.

\begin{itemize}
\item[(i)] $X$ creeps, viz. $\mathbb{P}(X_{T_{x}}=x)>0$ for some, and hence
for all, $x>0,$ if and only if $b>0.$ In that case, for any $0<t\leq \infty,$
the occupation measure 
\begin{equation*}
U_{t}(dy):=\mathbb{E}\left( \int_{0}^{t}ds1_{\{X_{s}\in dy\}}\right) ,\qquad
y\geq 0,
\end{equation*}%
has a continuous and bounded density on $(0,\infty ),$ $u_{t}(y),y>0.$ The
formula 
\begin{equation}
\mathbb{P}(T_{x}\in (t,t+\Delta ],X_{T_{x}}=x)=b\int_{[0,x)}\mathbb{P}%
(X_{t}\in dy)u_{\Delta }(x-y),  \label{creepinginterval}
\end{equation}%
holds for $x>0,$ $t\geq 0,$ $\Delta >0.$

\item[(ii)] On the event of non-creeping, $\{X_{T_{x}}>x\},$ the first
passage time distribution has a density given by 
\begin{equation*}
\mathbb{P}(T_{x}\in dt,X_{T_{x}}>x)=\left( \int_{[0,x)}\mathbb{P}(X_{t}\in
dy)\overline{\Pi }(x-y)\right) dt.
\end{equation*}
The formula 
\begin{equation}
\begin{split}
& \mathbb{P}(T_{x}\in (t,t+\Delta ],X_{T_{x}}>x) \\
& =\int_{[0,x)}\mathbb{P}(X_{t}\in dy)\int_{[0,x-y]}U_{\Delta }(dz)\overline{%
\Pi }(x-y-z),
\end{split}
\label{jumpinginterval}
\end{equation}%
holds for $x>0,$ $t\geq 0,$ $\Delta >0.$
\end{itemize}
\end{lemma}

\begin{proof}
The first claim in (i) in Lemma~\ref{CJF} is Kesten's \cite{kesten}
classical result, which states that when the drift $b>0,$ the potential measure $U_{\infty}$ is absolutely continuous with a continuous and bounded density $u_{\infty},$
\begin{equation*}
U_{\infty}(dy):=\mathbb{E}\left( \int_{0}^{\infty}ds1_{\{X_{s}\in dy\}}\right)=u_{\infty}(y)dy,\qquad y\geq 0,
\end{equation*}
and the following identity holds
\begin{equation}\label{kesten}\mathbb{P}(X_{T_{x}}=x)=bu_{\infty}(x),\qquad x>0.\end{equation} The absolute continuity of $U_{t}$ on $(0,\infty)$ and the bound $bu_{t}\leq 1,$ for any $t>0,$ follow from the fact that $U_{t}$ is dominated by above by $U_{\infty}$. We have furthermore the identity 
$$u_{t}(y)=u_{\infty}(y)-\int_{[0,y]}\mathbb{P}(X_{t}\in dz)u_{\infty}(y-z), \qquad y>0;$$ from where the continuity of $u_{t}$ in $(0,\infty)$ is deduced using the continuity of $u_{\infty}$ and the dominated convergence theorem.  

 In \cite{griffinmaller}, extending a result in \cite{winkel}, it has been proved that 
$$\mathbb{P}(X_{T_{x}}=x, T_{x}\leq t)=b\frac{\partial}{\partial y} U_{t}[0,y]|_{y=x}.$$ The latter together with Lebesgue's derivation theorem ensures that for a.e. $x>0$ $$\mathbb{P}(X_{T_{x}}=x, T_{x}\leq t)=bu_{t}(x).$$ Now, the claim will be obtained once we prove that for any $t>0,$ the function \begin{equation}\label{densityt}
x\mapsto\mathbb{P}(X_{T_{x}}=x, T_{x}\leq t),\qquad x>0,\end{equation} is continuous. An argument using the right-continuity and the quasi-left-continuity of $X$,  just as in the first paragraph of page 80 in~\cite{bertoinbook}, shows that for any fixed $t>0$ and $x>0,$ if $(x_{n}, n\geq 0)$ is a sequence that increases to $x,$ then we have both
$$\limsup_{n\to\infty}\mathbb{P}(X_{T_{x_{n}}}=x_{n}, X_{t}\geq x_{n})\leq \mathbb{P}(X_{T_{x}}=x, X_{t}\geq x),$$ and 
$$\limsup_{n\to\infty}\mathbb{P}(X_{T_{x_{n}}}=x_{n}, X_{t}< x_{n})\leq \mathbb{P}(X_{T_{x}}=x, X_{t}< x).$$ Similar inequalities are obtained if $(x_{n}, n\geq 0)$ is a sequence that decreases to $x.$ From the latter and the identity (\ref{kesten}) it is inferred that 
\begin{equation*}
\begin{split}
&\limsup_{n\to\infty}\mathbb{P}(X_{T_{x_{n}}}=x_{n}, X_{t}<x_{n})\\
&=\limsup_{n\to\infty}\left[\mathbb{P}(X_{T_{x_{n}}}=x_{n}) - \mathbb{P}(X_{T_{x_{n}}}=x_{n}, X_{t}\geq x_{n})\right]\\
&=\mathbb{P}(X_{T_{x}}=x)-\liminf_{n\to\infty}\mathbb{P}(X_{T_{x_{n}}}=x_{n}, X_{t}\geq x_{n})\\
&\leq \mathbb{P}(X_{T_{x}}=x, X_{t}<x)\\
&= \mathbb{P}(X_{T_{x}}=x)- \mathbb{P}(X_{T_{x}}=x, X_{t}\geq x).
\end{split}
\end{equation*}
Which leads to the inequality $$\mathbb{P}(X_{T_{x}}=x, X_{t}\geq x)\leq \liminf_{n\to\infty}\mathbb{P}(X_{T_{x_{n}}}=x_{n}, X_{t}\geq x_{n}).$$ The continuity of the function defined in (\ref{densityt}) is deduced therefrom. 

The proof of the formula (\ref{creepinginterval}) is obtained from the identity
\begin{equation*}
\begin{split}
&\mathbb{P}(T_{x}\in(t,t+\Delta], X_{T_{x}}=x)=\mathbb{P}(X_{t}<x, T_{x}\circ\theta_{t}\in(0,\Delta], X_{T_{x}}=x)\\
&=\int_{[0,x)}\mathbb{P}(X_{t}\in dy)\mathbb{P}(T_{x-y}\in(0,\Delta], X_{T_{x-y}}=x-y),
\end{split}
\end{equation*}
where $\theta_{t}$ denotes the shift operator, and we applied the simple Markov property at time $t$ to get the second equality.  Also, in the case where for any $s>0,$ $\mathbb{P}(X_{s}\in dy)=f_{s}(y)dy,$ with $f_{s}(y)$ a continuous function in $y$, we can take $$u_{t}(y)=\int^{t}_{0}ds f_{s}(y).$$ It follows from the formulas above that $T_{x}$ has a density on the event of creeping and $$\mathbb{P}(T_{x}\in dt, X_{T_{x}}=x)=bf_{t}(x),\qquad x>0, t>0.$$ 
The result in (ii) follows from the fact  
\begin{equation*}
h_{x}^{J}(t)=\int_{0}^{x}\mathbb{P}(X_{t}\in dy)\overline{\Pi }(x-y),\qquad t>0,
\end{equation*}
proved in \cite{doney-rivero}, together with an application of the Markov property as above.
\end{proof} 

Most of our calculations involve an exponential change of measure, which we
introduce now. For $\psi ^{\prime }(\infty )=b<\frac{x}{t}:=x_{t}<\mu =\psi
^{\prime }(0+)$ we denote by $(Y_{s},s\geq 0)$, a subordinator whose Laplace
exponent is 
\begin{equation}
\psi_{\rho_{t}}(\lambda )=\psi (\rho_{t}+\lambda )-\psi (\rho _{t})=b\lambda
+\int_{(0,\infty )}(1-e^{-\lambda y})e^{-\rho_{t}y}\Pi (dy),\qquad \lambda
\geq 0.
\end{equation}%
In particular we have the following relation: 
\begin{equation}
\mathbb{P}(Y_{t}\in dy)=e^{tH(\rho_{t})}e^{-\rho_{t}(y-tx_{t})}\mathbb{P}%
(X_{t}\in dy),\qquad y\in \mathbb{R}^{+}.  \label{abscont}
\end{equation}%
Observe that in the above definition of $Y$ we are deliberately excluding
the dependence in $x_{t}$ of $Y.$ We do this for notational convenience and
also because we will mainly use the equality of measures in (\ref{abscont}).

The proof of our main results rely on the following technical results. The
first of them relates various quantities we will consider.

\begin{lemma}
\label{lemma2} We have the following relations

\begin{itemize}
\item[(a)] $\displaystyle\frac{1}{2e}Q_{\Pi }(1/u)\leq H(u)\leq Q_{\Pi
}(1/u),$ for $u>0.$

\item[(b)] $u^{2}\sigma^{2}(u)\leq 2 H(u)$ for $u\geq 0.$

\item[(c)] $\displaystyle\frac{u^{2}\sigma ^{2}(u)}{H(u)}\geq \frac{e^{-1}}{%
\left( 1+\frac{\overline{\Pi }(1/u)}{K_{\pi }(1/u)}\right) },$ for $u>0.$ In
particular, if $X$ is stochastically compact at infinity, respectively at $%
0, $ then 
\begin{equation*}
\liminf \frac{u^{2}\sigma ^{2}(u)}{H(u)}>0,
\end{equation*}%
as $u\xrightarrow[]{}0,$ respectively as $u\xrightarrow[]{}\infty .$
\end{itemize}
\end{lemma}

\begin{proof}
Just proceed as in Lemma 5.1 in \cite{J+P}.
\end{proof}

\begin{lemma}
For $t>0,$ $b<x_{t}<\mu,$ we have for any $s>0$ 
\begin{align}  \label{lb4}
\mathbb{E}(Y_{s})& =sx_{t}=:\mu_{s}, \\
\mathbb{E}(Y_{s}-\mu_{s})^{2}& =s\int_{0}^{\infty }y^{2}e^{-\rho_{t}{y}}\Pi
(dy)=s\sigma ^{2}(\rho_{t}), \\
\mathbb{E}(Y_{s}-\mu_{s})^{3}& =s\int_{0}^{\infty }y^{3}e^{-\rho_{t}{y}}\Pi
(dy), \\
\mathbb{E}\left(|Y_{s}-\mu_{s}|^{3}\right)&\leq 6s(\rho_{t})^{-3}Q_{\Pi
}(1/\rho_{t})+2\mu _{s}s\sigma ^{2}(\rho_{t}).
\end{align}
\end{lemma}

\begin{proof}
The first three identities are proved by bare hands calculations, while the
claimed upper bound is obtained as follows: 
\begin{equation*}
\begin{split}
\mathbb{E}\left(|Y_{s}-\mu_{s}|^{3}\right)& =\mathbb{E}\left((Y_{s}-\mu_{s})^{3}\right)+2\mathbb{E}\left( (\mu
_{s}-Y_{s})^{3}:Y_{s}\leq \mu_{s}\right) \\
& \leq \mathbb{E}\left((Y_{s}-\mu_{s})^{3}\right)+2\mu_{s}\mathbb{E}\left( (\mu_{s}-Y_{s})^{2}:Y_{s}\leq
\mu_{s}\right) \\
& \leq s\int_{0}^{\infty }y^{3}e^{-\rho_{t}{y}}\Pi (dy)+2\mu_{s}s\sigma
^{2}(\rho_{t}) \\
& =s\int_{\{y\rho_{t}\leq 1\}}y^{3}e^{-\rho_{t}{y}}\Pi (dy)+s\int_{\{y\rho
_{t}>1\}}y^{3}e^{-\rho_{t}{y}}\Pi (dy)+2\mu_{s}s\sigma ^{2}(\rho_{t}) \\
& \leq s(\rho_{t})^{-1}\int_{\{y\rho_{t}\leq 1\}}y^{2}\Pi (dy)+6s(\rho
_{t})^{-3}\overline{\Pi }(1/\rho_{t})+2\mu_{s}s\sigma ^{2}(\rho_{t}) \\
& \leq 6s(\rho_{t})^{-3}Q_{\Pi }(1/\rho_{t})+2\mu_{s}s\sigma ^{2}(\rho
_{t}).
\end{split}
\label{estimatefor3}
\end{equation*}\end{proof}

\begin{lemma}
\label{uniformityH} In the settings ($SC_{0}$-(I-II)), $(SC_{\infty})$ and $%
(G),$ we have that $tH(\rho_{t})\to\infty$ uniformly in $x$.
\end{lemma}

\begin{proof}
The proof of the case $(G)$ is a straightforward consequence of the fact
that in this setting $t\rightarrow \infty $ and 
\begin{equation*}
0<\liminf_{t\rightarrow \infty }H(\rho_{t})\leq \limsup_{t\rightarrow
\infty }H(\rho_{t})<\infty ,
\end{equation*}%
because $b<\liminf_{t\to\infty} x_{t}\leq \limsup_{t\to\infty} x_{t}<\mu,$ and hence $$0<\liminf_{t\to\infty}\rho_{t}\leq \limsup_{t\to\infty}\rho_{t}<\infty.$$ To deal with the cases ($SC_{0}$-(I-II)), $(SC_{\infty })$ we
use Theorem 5.1 of Jain and Pruitt \cite{J+P} which establishes that the
condition $tH(\rho_{t})\rightarrow \infty $ is equivalent to $\mathbb{P}(X_{t}\leq
x)\rightarrow 0,$ as $t\rightarrow \infty $ or $0.$ For the case $(SC_{0}-I)$
when (\ref{SC00}) fails, the equality 
\begin{equation*}
\mathbb{P}(X_{t}\leq x)=\mathbb{P}\left( \frac{X_{t}}{t}-b\leq x_{t}-b\right) ,
\end{equation*}%
and an application of the weak law of large numbers for subordinators gives
the result. To deal with the case $(SC_{0}-II),$ we use the equality 
\begin{equation*}
\mathbb{P}(X_{t}\leq x)=\mathbb{P}\left( \frac{X_{t}-\mathbf{b}(t)}{c(t)}\leq \frac{x-\mathbf{b
}(t)}{c(t)}\right) ,
\end{equation*}%
which together with the sequential convergence in (\ref{seqlimit}) and the
assumption that $\frac{x-\mathbf{b}(t)}{c(t)}\rightarrow -\infty $ as $%
t\rightarrow 0,$ lead to $\mathbb{P}(X_{t}\leq x)\rightarrow 0$ as $t\rightarrow 0.$
The case when (\ref{SC00}) holds as well as the cases $(SC_{\infty })$ are
proved with a similar argument. To show the uniformity observe that the
function 
\begin{equation*}
\lambda \mapsto H(\lambda )=\psi (\lambda )-\lambda \psi ^{\prime }(\lambda
)=\int_{(0,\infty )}(1-e^{-\lambda x}-\lambda xe^{-\lambda x})\Pi (dx),
\end{equation*}%
is increasing because the function $z\mapsto 1-e^{-z}-ze^{-z}$ is. This
implies that the function $\lambda \mapsto H(\rho (\lambda ))$ is
decreasing. The uniformity in the cases $(G)$ and ($SC_{0}$-(I-II)) follows
easily from this fact. Indeed, it is enough to observe that $tH(\rho_{t})$
tends towards $\infty $ as soon as we take a $x_{0}$ such that $x_{0}>x$ and 
$tH\left( \rho \left( \frac{x_{0}}{t}\right) \right) \rightarrow \infty .$
To establish the uniformity in the case $(SC_{\infty })$ when (\ref{SC00})
holds, we observe that the hypotheses imply that there is a function $D$
such that $x\leq \mathbf{b}(t)-D(t),$ and $D(t)/c(t)\rightarrow \infty $ as $%
t\rightarrow \infty .$ The function $D$ is such that 
\begin{equation*}
tH(\rho_{t})\geq tH\left( \rho \left( \frac{\mathbf{b}(t)-D(t)}{t}\right)
\right) \xrightarrow[t\to\infty]{}\infty ,
\end{equation*}%
because by the assumption of stochastic compactness at $\infty $ we have
that 
\begin{equation*}
\mathbb{P}(X_{t}<\mathbf{b}(t)-D(t))=\mathbb{P}\left( \frac{X_{t}-\mathbf{b}(t)}{c(t)}\leq -%
\frac{D(t)}{c(t)}\right) \xrightarrow[t\to\infty]{}0.
\end{equation*}%
In the case $(SC_{\infty }),$ when (\ref{SC00}), does not hold we proceed as
above but using that there is a function $j$ such that $x\leq bt+j(t)$ and $%
j(t)/c(t)\rightarrow 0$ as $t\rightarrow \infty .$
\end{proof}

\begin{lemma}
\label{lemma02} If (H) holds then 
\begin{equation*}
\begin{split}
& \int_{|z|>1}\left\vert \exp \{-t\left( \psi (\lambda -iz)-\psi (\lambda
)\right) \}\right\vert \left\vert \frac{\psi_{*}(\lambda -iz)}{\lambda -iz}%
\right\vert dz \\
& \leq e^{2t\psi_{*}(\lambda )}\int_{|z|>1}\left\vert \exp \{-t\left( \psi
(-iz)\right) \}\right\vert \frac{\psi_{*}(\lambda )+\left\vert \psi
_{*}(-iz)\right\vert }{z}dz<\infty ,
\end{split}%
\end{equation*}%
for any $\lambda >0$ and $t>{t}_{0}.$
\end{lemma}

\begin{proof}The proof of this result is an easy consequence of the two inequalities:
\begin{equation*}
\begin{split}
& \mathcal{R}\left( \psi (\lambda -iz)-\psi (\lambda )\right) \\
& =\int_{(0,\infty )}(1-\cos (zy))e^{-\lambda y}\Pi (dy) \\
& =\int_{(0,\infty )}(1-\cos (zy))\Pi (dy)-\int_{(0,\infty )}(1-\cos
(zy))\left( 1-e^{-\lambda y}\right) \Pi (dy) \\
& \geq \int_{(0,\infty )}(1-\cos (zy))\Pi (dy)-2\int_{(0,\infty )}\left(
1-e^{-\lambda y}\right) \Pi (dy) \\
& =\mathcal{R}\psi (-iz)-2\int_{(0,\infty )}\left( 1-e^{-\lambda y}\right)
\Pi (dy);
\end{split}%
\end{equation*}%
and
\begin{equation*}
\begin{split}
& |\left( \psi_{*}(\lambda -iz)-\psi_{*}(-iz)\right) | \\
& =|\int_{(0,\infty )}(1-e^{-\lambda y})e^{izy}\Pi (dy)|\leq \int_{(0,\infty
)}(1-e^{-\lambda y})\Pi (dy)=\psi_{*}(\lambda ).
\end{split}
\end{equation*}
\end{proof}

\section{Proof of Theorem~\protect\ref{thlocalXSC0} and Theorem~\protect\ref%
{thm01} in the $SC_{0}$ cases}

\subsection{Proof of (i) in Theorem~\protect\ref{thlocalXSC0}}

Let $U(x)=\int_{0}^{x}y\overline{\Pi }(y)dy,$ $x\geq 0,$ and 
\begin{equation*}
\widehat{U}(s)=\int_{0}^{\infty }e^{-sy}U(dy)=s\int_{0}^{\infty
}e^{-sy}U(y)dy,\quad s>0,
\end{equation*}
be its Laplace transform. By hypothesis we have that the condition ($%
SC^{\prime}_{0}$) is satisfied, which implies that $U$ has bounded increase
at $0,$ see \cite{BGT} page 68. So by Theorem 2.10.2 in \cite{BGT}, actually
from its proof, we know that there are constants $0<c_{1}\leq c_{2}<\infty $
such that for small $s,$ $c_{1}U(s)\leq \widehat{U}(1/s)\leq c_{2}U(s).$
From this it follows that $\widehat{U}$ has bounded decrease at infinity.
Indeed, taking $\alpha $ as in ($SC_{0}^{\prime }$) we have for $\lambda >1$ 
\begin{equation*}
\liminf_{s\rightarrow \infty }\frac{\widehat{U}(s\lambda )}{\widehat{U}(s)}%
\geq c_{3}\liminf_{s\rightarrow \infty }\frac{U\left( \frac{1}{s\lambda }%
\right) }{U\left( \frac{1}{s}\right) }=c_{3}\left( \limsup_{v\rightarrow 0+}%
\frac{U(\lambda v)}{U(v)}\right) ^{-1}\geq c_{3}\frac{1}{\lambda ^{2-\alpha }%
}.
\end{equation*}
Proposition 2.2.1 in \cite{BGT} implies that for any $\beta <-(2-\alpha )$
there exist constants $c_{4}>0$ and $\widetilde{\ell }>0,$ such 
\begin{equation}
\frac{\widehat{U}(y)}{\widehat{U}(x)}\geq c_{4}\left( \frac{y}{x}\right)
^{\beta },\qquad y\geq x\geq \widetilde{\ell }.  \label{eqLTU}
\end{equation}%
Also, an easy integration by parts implies the identity 
\begin{equation}
\widehat{U}(s)=-\left( \frac{\psi (s)}{s}\right) ^{\prime }=\frac{H(s)}{s^{2}%
},\qquad s>0.
\end{equation}%
So, by (\ref{eqLTU}) we have 
\begin{equation}  \label{eq:29}
\frac{H(y)}{H(x)}\geq c_{4}\left( \frac{y}{x}\right) ^{2+\beta },\qquad
y\geq x\geq \widetilde{\ell },
\end{equation}%
where $2+\beta <\alpha \leq 2.$ Since $0<\alpha $ there exists $\beta_{0}$
and positive constants $c_{6}$ and $\widetilde{\ell }$ such that $0<2+\beta
_{0}<\alpha \leq 2$ and 
\begin{equation*}
H(y)\geq y^{2+\beta_{0}}c_{6},\qquad y\geq \widetilde{\ell }.
\end{equation*}%
To conclude we observe that the following inequalities hold 
\begin{equation}  \label{eq:30}
\begin{split}
\int_{0}^{\infty }(1-\cos (\theta y))\Pi (dy)& \geq c_{7}\theta
^{2}\int_{0}^{1/\theta }y^{2}\Pi (dy) \\
& =c_{7}K_{\Pi }(1/\theta )=c_{7}\frac{K_{\Pi }(1/\theta )}{Q_{\Pi
}(1/\theta )}Q_{\Pi }(1/\theta ) \\
& \geq c_{8}Q_{\Pi }(1/\theta ) \\
& =2c_{8}\theta ^{2}U(1/\theta ) \\
& \geq c_{9}H(\theta )
\end{split}%
\end{equation}%
for $\theta $ large enough; here we used the assumption $(SC_{0})$ and the
equality (\ref{qp}). We infer that for $\theta >0$ large enough 
\begin{equation*}
\Re (\psi (i\theta ))=\int_{0}^{\infty }(1-\cos (\theta y))\Pi (dy)\geq
c_{10}\theta ^{\beta_{0}+2}.
\end{equation*}%
As a consequence, for $n\geq 0$ 
\begin{equation*}
\int_{\mathbb{R}}|\theta |^{n}|\mathbb{E}(e^{i\theta X_{t}})|d\theta =\int_{%
\mathbb{R}}|\theta |^{n}\exp \{-t\Re (\psi (i\theta ))\}d\theta <\infty ,
\end{equation*}%
and the conclusion follows from standard results.

\subsection{Proof of (ii) in Theorem~\protect\ref{thlocalXSC0}}

Before we start with the proof we state a further auxiliary theorem. This is
a consequence of Lemma 1, P109, in \cite{petrov}.

\begin{lemma}
\label{Petrov+D} Let $Z_{1,}Z_{2},\cdots Z_{n}$ be independent rvs having
finite 3rd moments, write $\mathbb{E}(Z_{r})=\mu_{r},Var(Z_{r})=%
\sigma_{r}^{2},$ and $\mathbb{E}\left(|Z_{r}-\mu_{r}|^{3}\right)=\nu_{r},$
and put $W=\sum_{_{1}}^{n}Z_{r},m=\mathbb{E}(W)=\sum_{_{1}}^{n}\mu_{r},$ and 
$s^{2}=VarW=\sum_{_{1}}^{n}\sigma_{r}^{2}$. Assume further that $%
\int_{-\infty }^{\infty }|\Psi (u)|du<\infty ,$ where $\Psi (u)=\mathbb{E}%
\left(e^{iuW}\right),$ and denote by $f$ and $\phi $ the pdf of $W$ and the
standard Normal pdf$.$ Then there is an absolute constant $A$ such that%
\begin{equation}
\sup_{y}\left|f(y)-s^{-1}\phi \left(\frac{y-m}{s}\right)\right|\leq AL+d,
\label{2}
\end{equation}%
where $\ L=\sum_{1}^{n}\nu_{r}/s^{4}$ and, with $l=(4Ls^{2})^{-1},$%
\begin{equation*}
d=2\int_{l}^{\infty }|\Psi (u)|du.
\end{equation*}
\end{lemma}

\begin{proof}
Essentially same as Lemma 3 of \cite{rad}.
\end{proof}

\begin{remark}
Our use of this result exploits the fact that, for any L\'{e}vy process, any 
$t>0,$ and any $n\geq 1,$ $X_{t}$ is the sum of $n$ independent and
identically distributed summands.
\end{remark}

\begin{proof}[Proof of (ii) in Theorem~\protect\ref{thlocalXSC0}]
We observe first that the assumption that $b<x_{t}<\mu $ and $%
x_{t}\rightarrow b$ implies that $\rho_{t}\rightarrow \infty ,$
irrespective of whether $t\rightarrow 0$ or $t\rightarrow \infty .$ We next
establish that these conditions on $x_{t}$, the fact that $tH(\rho
_{t})\rightarrow \infty ,$ and the stochastic compactness at $0,$ imply that 
$x\rho_{t}\rightarrow \infty ,$ again irrespective of whether $t\rightarrow
0$ or $t\rightarrow \infty $. Indeed, the identities 
\begin{equation}
\frac{tH(\rho_{t})}{x\rho_{t}}=\frac{t\psi (\rho_{t})-t\rho_{t}\psi
^{\prime }(\rho_{t})}{t\rho_{t}\psi ^{\prime }(\rho_{t})}=\frac{\psi
(\rho_{t})}{\rho_{t}\psi ^{\prime }(\rho_{t})}-1,  \label{eq34}
\end{equation}%
show that it is enough to justify that $0<\liminf_{z%
\rightarrow \infty }\frac{z\psi ^{\prime }(z)}{\psi (z)}.$ If the drift of $%
X $ is positive this is straightforward. If the drift is zero this holds
whenever $\limsup_{z\rightarrow 0}\frac{z\overline{\Pi }(z)}{%
\int_{0}^{z}y\Pi (dy)}<\infty ,$ which in turn holds by stochastic
compactness at zero, 
\begin{equation*}
\limsup_{z\rightarrow 0}\frac{z\overline{\Pi }(z)}{\int_{0}^{z}y\Pi (dy)}%
\leq \limsup_{z\rightarrow 0}\frac{z^{2}\overline{\Pi }(z)}{%
\int_{0}^{z}y^{2}\Pi (dy)}<\infty .
\end{equation*}%
The former claim is an easy consequence of the following inequalities 
\begin{equation*}
\begin{split}
\frac{\lambda \psi ^{\prime }(\lambda )}{\psi (\lambda )}& =\frac{%
\int_{0}^{\infty }ye^{-\lambda y}\Pi (dy)}{\int_{0}^{\infty }e^{-\lambda y}%
\overline{\Pi }(y)dy} \\
& \geq \frac{\int_{0}^{1/\lambda }ye^{-1}\Pi (dy)}{\int_{0}^{1/\lambda }%
\overline{\Pi }(y)dy+(1/\lambda )\overline{\Pi }(1/\lambda )} \\
& =\frac{e^{-1}\int_{0}^{1/\lambda }y\Pi (dy)}{\int_{0}^{1/\lambda }y{\Pi }%
(dy)+(2/\lambda )\overline{\Pi }(1/\lambda )}
\end{split}%
\end{equation*}%
which are obtained by barehand calculations. It is important to remark that
the above facts and the Lemma~\ref{uniformityH} imply that $x\rho
_{t}\rightarrow \infty $ uniformly in $x.$ Furthermore, our previous remarks
allow us to provide a unified proof of the cases $t\rightarrow 0$ or $%
t\rightarrow \infty .$

We will apply Lemma~\ref{Petrov+D} with $n=[x\rho_{t}]$ and $%
W=Y_{t}=\sum_{k=1}^{n}Z_{k}$ with $Z_{k}=Y_{\frac{tk}{n}}-Y_{\frac{t(k-1)}{n}%
}\overset{\text{Law}}{=}Y_{\frac{t}{n}},$ for $k\in \{1,\ldots ,n\}.$ We use
the estimate (\ref{lb4}) with $s=t/n,$ thus $\mu_{s}=x/n,$ which together
with our choice of $n$ lead to the approximation 
\begin{equation}\label{eq:392}
\begin{split}
n\nu & :=n\mathbb{E}\left(|Z_{1}-x/n|^{3}\right)\leq t\left\{ 6(\rho_{t})^{-3}Q_{\Pi }(1/\rho
_{t})+2\frac{x}{n}\sigma ^{2}(\rho_{t})\right\} \\
& = t\left\{ 6(\rho_{t})^{-3}Q_{\Pi }(1/\rho_{t})+2\frac{\rho
_{t}^{2}\sigma ^{2}(\rho_{t})}{\rho_{t}^{3}}\left(1+o(1)\right)\right\} ,
\end{split}%
\end{equation}%
for $\rho_{t}$ large enough; here the term 
$$1\leq 1+o(1)=\frac{x\rho_{t}}{[x\rho_{t}]}\leq \frac{1}{1-\frac{1}{x\rho_{t}}}\to 1,$$ and the convergence is uniform in $x.$ It is then immediate from the definition of $L$
that for $\rho_{t}$ large enough 
\begin{equation}\label{eq:393}
\sqrt{t}\sigma (\rho_{t})L=\frac{n\nu }{\{t\sigma ^{2}(\rho_{t})\}^{\frac{3%
}{2}}}\leq \frac{1}{\left( t\rho_{t}^{2}\sigma ^{2}(\rho_{t})\right)
^{1/2}}\left( 6\frac{Q_{\Pi }(1/\rho_{t})}{\rho_{t}^{2}\sigma ^{2}(\rho
_{t})}+2(1+o(1))\right).
\end{equation}%
Which because of the assumption of stochastic compactness at $0$ and Lemma~%
\ref{lemma2} imply that for $x\rho_{t}$ and $\rho_{t}$ large enough there is a constant $k_{1}$ such that 
\begin{equation}\label{eq:394}
\sqrt{t}\sigma (\rho_{t})L\leq \frac{k_{1}}{\sqrt{tH(\rho_{t})}}.
\end{equation}%
So the lemma tells us that (\ref{6a}) holds provided that 
\begin{equation*}
\gamma :=\sqrt{t}\sigma (\rho_{t})\int_{l}^{\infty }e^{-t\Re (\psi_{\rho
_{t}}(i\theta ))}d\theta \rightarrow 0.
\end{equation*}%
To prove that this is indeed the case, observe that the above estimate for $L$
gives 
\begin{equation*}
\ell =\frac{1}{4Ls^{2}}=\frac{t\sigma ^{2}(\rho_{t})}{4n\nu }\gtrsim k_{2}%
\frac{\left( t\rho_{t}^{2}\sigma ^{2}(\rho_{t})\right) ^{1/2}}{\left(
t\sigma ^{2}(\rho_{t})\right) ^{1/2}}=k_{2}\rho_{t},
\end{equation*}%
for $\rho_{t}$ large enough. Applying the inequality (\ref{eq:30}) we
obtain that for $\theta \geq \ell \geq k_{2}\rho_{t}$ 
\begin{equation*}
\begin{split}
\Re (\psi_{\rho_{t}}(i\theta ))& =\int_{0}^{\infty }(1-\cos (\theta
y))e^{-\rho_{t}y}\Pi (dy) \\
& \geq k_{3}e^{-1/k_{2}}\theta ^{2}\int_{0}^{1/\theta }y^{2}\Pi (dy)\geq
k_{4}H(\theta ).
\end{split}%
\end{equation*}%
It follows from the above and the estimate (\ref{eq:29}) that for any $%
0<\alpha_{0}<\alpha \leq 2,$ with $\alpha $ as in $(SC_{0}^{\prime }),$ and
for $\rho_{t}$ large enough 
\begin{equation*}
\begin{split}
\sqrt{t}\sigma (\rho_{t})\int_{l}^{\infty }e^{-t{\Re }(\psi_{\rho
}(i\theta ))}d\theta & \leq \sqrt{t}\sigma (\rho_{t})\int_{l}^{\infty }\exp
\{-k_{4}tH(\rho_{t})\frac{H(\theta )}{H(\rho_{t})}\}d\theta \\
& \leq \sqrt{t}\sigma (\rho_{t})\int_{\ell }^{\infty }\exp \left\{
-k_{5}tH(\rho_{t})\left( \frac{\theta }{\rho_{t}}\right) ^{\alpha
_{0}}\right\} d\theta \\
& \leq k_{6}\sqrt{t}\rho_{t}\sigma (\rho_{t})\int_{k_{2}}^{\infty }\exp
\left\{ -k_{4}tH(\rho_{t})\theta ^{\alpha_{0}}\right\} d\theta \\
& \leq k_{7}\frac{\sqrt{t}\rho_{t}\sigma (\rho_{t})}{\left( tH(\rho
_{t})\right) ^{1/\alpha_{0}}} \\
& \leq k_{8}\frac{1}{\left( tH(\rho_{t})\right) ^{1/\alpha_{0}-1/2}}%
\xrightarrow[]{}0,
\end{split}%
\end{equation*}%
where in the last inequality we used Lemma 7. Observe that the uniformity
follows from Lemma~\ref{uniformityH} and the fact that $\rho_{t}$ tends to
infinity uniformly as well because it is non-increasing.
\end{proof}

\subsection{Proof of Theorem~\protect\ref{thm01} in the $SC_{0}$ cases}

\begin{proof}[Proof of estimate in (\protect\ref{eq01}) in the $SC_{0}$ case]
With $s_{t}:=\sqrt{t}\sigma (\rho_{t})$ we start by observing that
\begin{equation}
\begin{split}
& \sqrt{2\pi }s_{t}e^{tH(\rho_{t})}h_{x}^{J}(t) \\
& =\sqrt{2\pi }s_{t}e^{tH(\rho_{t})}\int_{0}^{x}\overline{\Pi }%
(y)f_{t}(x-y)dy \\
& =\sqrt{2\pi }\int_{0}^{x}\overline{\Pi }(y)e^{-\rho_{t}y}\left( \phi \left(
\frac{y}{s_{t}}\right)+o(1)\right) dy \\
& \leq \frac{\psi_{*}(\rho_{t})}{\rho_{t}}(1+o(1))
\end{split}%
\end{equation}%
where we have used the identity%
\begin{equation*}
\int_{0}^{\infty }e^{-\rho_{t}y}\overline{\Pi }(y)dy=\frac{\psi_{*}(\rho_{t})}{%
\rho_{t}}.
\end{equation*}
To establish a lower bound, we use that for $\varepsilon >0,$ there exists a 
$\delta >0$ such that if $v\leq \delta s_{t}$ we have $\sqrt{2\pi }\phi
\left( \frac{v}{s_{t}}\right) \geq 1-\varepsilon .$ Put $x^{\ast }:=x\wedge
\delta s_{t}$ and write 
\begin{eqnarray*}
\sqrt{2\pi }\int_{0}^{x}\overline{\Pi }(y)e^{-\rho_{t}y}\phi \left(\frac{y}{s_{t}%
}\right)dy &\geq &(1-\varepsilon )\int_{0}^{x^{\ast }}\overline{\Pi }(y)e^{-\rho
_{t}y}dy \\
&=&(1-\varepsilon )\left( \frac{\psi_{*}(\rho_{t})}{\rho_{t}}-\int_{x^{\ast
}}^{\infty }\overline{\Pi }(y)e^{-\rho_{t}y}dy\right) \\
&=&(1-\varepsilon )\left( \frac{\psi_{*}(\rho_{t})}{\rho_{t}}\right)
(1+o(1)).
\end{eqnarray*}%
Here we use%
\begin{eqnarray*}
\frac{\int_{x^{\ast }}^{\infty }\overline{\Pi }(y)e^{-\rho_{t}y}dy}{\frac{%
\psi (\rho_{t})}{\rho_{t}}-b} &\leq &\frac{\overline{\Pi }(x^{\ast
})e^{-\rho_{t}x^{\ast }}}{\int_{0}^{\infty }(1-e^{-\rho_{t}y})\Pi (dy)} \\
&\leq &\frac{\overline{\Pi }(x^{\ast })e^{-\rho_{t}x^{\ast }}}{(1-e^{-\rho
_{t}x^{\ast }})\overline{\Pi }(x^{\ast })}=\frac{1}{e^{\rho_{t}x^{\ast }}-1}
\end{eqnarray*}%
and the fact that $x^{\ast }\rho_{t}\rightarrow \infty ,$ uniformly either
as $t\rightarrow \infty $ or $t\rightarrow 0,$ which follows from the fact
that $s_{t}\rho_{t}=O(tH(\rho_{t})),$ Lemma \ref{uniformityH}; and that $x\rho_{t}\rightarrow\infty ,$ uniformly, which was proved in (\ref{eq34}). \end{proof}

Essentially the same arguments, together with the formulas in (\ref%
{creepinginterval}) and (\ref{jumpinginterval}) allow us to prove the
estimates in (\ref{eq03}) and (\ref{eq04}).

\begin{proof}[Proof of (\protect\ref{eq04}) in the $SC_{0}$ case]
From the identity (\ref{creepinginterval}) and (\ref{6a}) we deduce the
following upper bound 
\begin{equation*}
\begin{split}
& \sqrt{2\pi }s_{t}e^{tH(\rho_{t})}h_{x}^{C}(t,\Delta ) \\
& =b\sqrt{2\pi }s_{t}e^{tH(\rho_{t})}\int_{[0,x)}dyf_{t}(y)u_{\Delta }(x-y)
\\
& =b\int_{[0,x)}dy\left[ \sqrt{2\pi }\phi \left( \frac{y}{s_{t}}\right) +o(1)%
\right] e^{-\rho_{t}y}u_{\Delta }(y) \\
& \leq b(1+o(1))\int_{[0,\infty )}dye^{-\rho_{t}y}u_{\Delta }(y) \\
& =b(1+o(1))\frac{1-e^{-\Delta \psi (\rho_{t})}}{\psi (\rho_{t})}.
\end{split}%
\end{equation*}%
To get a lower bound, we proceed as in the previous proof and lower bound
the expression in the 3rd line above by%
\begin{equation*}
\begin{split}
& b(1-\varepsilon +o(1))\int_{[0,x^{\ast })}dye^{-\rho_{t}y}u_{\Delta }(y)
\\
& =b(1-\epsilon +o(1))\left( \frac{1-e^{-\Delta \psi (\rho_{t})}}{\psi
(\rho_{t})}-\int_{[x^{\ast },\infty )}dye^{-\rho_{t}y}u_{\Delta
}(y)\right) .
\end{split}%
\end{equation*}%
The conclusion follows from the bound

\begin{eqnarray*}
\int_{\lbrack x^{\ast },\infty )}dye^{-\rho_{t}y}u_{\Delta }(y) &\leq
&e^{-\rho_{t}x^{\ast }/2}\int_{[0,\infty )}dye^{-\rho_{t}y/2}u_{\Delta }(y)
\\
&=&e^{-\rho_{t}x^{\ast }/2}\frac{1-e^{-\Delta \psi (\rho_{t}/2)}}{\psi
(\rho_{t}/2)}=o\left( \frac{1-e^{-\Delta \psi (\rho_{t})}}{\psi (\rho_{t})%
}\right) ,
\end{eqnarray*}%
where we have used the fact that $\psi (\rho_{t})\backsim b\rho_{t},$ which is in turn a well known property of the Laplace exponent $\psi$ and the fact that $\rho_{t}\to\infty.$
\end{proof}

\begin{proof}[ Proof of (\protect\ref{eq03}) in the $SC_{0}$ case]
From the identity (\ref{jumpinginterval}) and (\ref{6a}) we deduce the
following upper bound 
\begin{equation*}
\begin{split}
& \sqrt{2\pi }s_{t}e^{tH(\rho_{t})}h_{x}^{J}(t,\Delta ) \\
& =\sqrt{2\pi }s_{t}e^{tH(\rho
_{t})}\int_{[0,x)}dyf_{t}(x-y)\int_{[0,y)}U_{\Delta }(dz)\overline{\Pi }(y-z)
\\
& =\int_{[0,x)}dy\left[ \sqrt{2\pi }\phi \left( \frac{y}{s_{t}}\right) +o(1)%
\right] e^{-\rho_{t}y}\int_{[0,y)}U_{\Delta }(dz)\overline{\Pi }(y-z) \\
& \leq \left[ 1+o(1)\right] \int_{[0,\infty )}dye^{-\rho
_{t}y}\int_{[0,y)}U_{\Delta }(dz)\overline{\Pi }(y-z) \\
& =(1+o(1))\left[ \int_{0}^{\Delta }ds\mathbb{E}(e^{-\rho_{t}X_{s}})\right] \left[
\int_{(0,\infty )}dz\overline{\Pi }(z)e^{-\rho_{t}z}\right]  \\
& =(1+o(1))\left[ \frac{1-e^{-\Delta \psi (\rho_{t})}}{\psi (\rho_{t})}%
\right] \left[ \frac{\psi_{*}(\rho_{t})}{\rho_{t}}\right] .
\end{split}%
\end{equation*}%
To get a lower bound, we proceed as in the previous proof and lower bound
the expression in the 3rd line above by 
\begin{equation*}
\begin{split}
&(1-\varepsilon +o(1))\int_{[0,x^{\ast })}dye^{-\rho
_{t}y}\int_{[0,y)}U_{\Delta }(dz)\overline{\Pi }(y-z) \\
&=(1-\varepsilon +o(1))\left( 
\left[ \frac{1-e^{-\Delta \psi (\rho_{t})}}{\psi (\rho_{t})}\right] \left[ 
\frac{\psi_{*}(\rho_{t})}{\rho_{t}}\right] \right.\\
&\quad\left.-\int_{[x^{\ast },\infty )}dye^{-\rho_{t}y}\int_{[0,y)}U_{\Delta }(dz)%
\overline{\Pi }(y-z)%
\right) .
\end{split}
\end{equation*}
The conclusion follows from the bound
\begin{equation*}
\begin{split}
&\int_{\lbrack x^{\ast },\infty )}dye^{-\rho_{t}y}\int_{[0,y)}U_{\Delta }(dz)
\overline{\Pi }(y-z) \\
&\leq e^{-\rho_{t}x^{\ast }/2}\int_{[0,\infty)}dye^{-\rho_{t}y/2}\int_{[0,y)}U_{\Delta }(dz)\overline{\Pi }(y-z) \\
&=e^{-\rho_{t}x^{\ast }/2}\left[ \frac{1-e^{-\Delta \psi (\rho_{t}/2)}}{%
\psi (\rho_{t}/2)}\right] \left[ \frac{\psi_{*}(\rho_{t}/2)}{%
\rho_{t}/2}\right] ,
\end{split}
\end{equation*}%
since it is easy to see that there is some $K$ with
\begin{equation*}
\left[ \frac{1-e^{-\Delta \psi (\rho_{t}/2)}}{\psi (\rho_{t}/2)}\right] %
\left[ \frac{\psi_{*}(\rho_{t}/2)}{\rho_{t}/2}\right] \leq K%
\left[ \frac{1-e^{-\Delta \psi (\rho_{t})}}{\psi (\rho_{t})}\right] \left[ 
\frac{\psi_{*}(\rho_{t})}{\rho_{t}}\right].
\end{equation*}
Indeed, this is a consequence of the results in Lemma~2 of Chapter 2 in \cite{doneysbook} and the fact that the function $x\mapsto U_{\Delta}[0,x]$ is subadditive, which in turn follows from the identity 
$$U_{\Delta}[0,x]=\mathbb{E}(T_{x}\wedge \Delta),$$ and the strong Markov property of $X$. \end{proof}

\section{Proof of Theorem~\protect\ref{thm01}}

We have already proved Theorem~\ref{thm01} in the $(SC_{0})$ cases, so we
will hereafter omit that case. For the $(SC_{\infty })$ and $(G)$ cases we
will use the following result instead of Lemma \ref{Petrov+D}.

\begin{lemma}
\label{lemma:06} Let $\Lambda $ be defined by 
\begin{equation*}
\Lambda :=s_{t}L=\frac{[x\rho_{t}]\mathbb{E}\left(|Y_{1}-\mu_{1}|^{3}\right)%
}{(t\sigma ^{2}(\rho_{t}))^{3/2}}.
\end{equation*}%
We have the inequality 
\begin{equation*}
|\mathbb{E}(e^{-iz(Y_{t}-\mu_{t})/s_{t}})-e^{-z^{2}/2}|\leq 16\Lambda
|z|^{3}e^{-z^{2}/3}\text{ for all }|z|\leq 1/4\Lambda ,
\end{equation*}%
and under the assumptions of Theorem~\ref{thm01} there exists a constant $%
K\in (0,\infty ),$ such that for large $t,$ 
\begin{equation}
\Lambda \leq \frac{K}{\sqrt{tH(\rho_{t})}}\rightarrow 0.  \label{Lambda}
\end{equation}%
uniformly in $x.$
\end{lemma}

\begin{proof}[Proof]
The claimed inequality is a consequence of the Ess\'een-like inequality in Lemma 1,
page 109 in \cite{petrov},  applied to the sum of i.i.d. random variables $$\sum^{n}_{k=1}\left(Y_{\frac{tk}{n}}-Y_{\frac{t(k-1)}{n}}\right)=Y_{t},$$ with $n=[x\rho_{t}].$ The common mean and variance are $\mu_{t/n}=\frac{t}{n}x_{t}$ and $s^{2}_{t/n}=\frac{t}{n}\sigma^{2}(\rho_{t}),$ respectively. As in the proof of (ii) in Theorem~\ref{thlocalXSC0} it is proved that $x\rho_{t}\to\infty$ uniformly. Using this fact it is easy to verify that the arguments used to obtain (\ref{eq:392}) and (\ref{eq:393}) can be extended to show that
\begin{equation}\label{eq:393bis}
\sqrt{t}\sigma (\rho_{t})L=\frac{[x\rho_{t}]\mathbb{E}\left(|Y_{1}-\mu_{1}|^{3}\right)}{(t\sigma
^{2}(\rho_{t}))^{3/2}}\lesssim \frac{1}{\left( t\rho_{t}^{2}\sigma ^{2}(\rho_{t})\right)
^{1/2}}\left( 6\frac{Q_{\Pi }(1/\rho_{t})}{\rho_{t}^{2}\sigma ^{2}(\rho
_{t})}+2\right).
\end{equation}
In the $(SC_{\infty})$ case, the above inequality together with the Lemma~\ref{lemma2} gives the result. In the $(G)$ case, the result also follows immediately from the latter inequality but one needs to recall that because $b<\liminf_{t\to\infty} x_{t}\leq \limsup_{t\to\infty} x_{t}<\mu,$ then
 $$0<\liminf_{t\to\infty}\rho_{t}\leq \limsup_{t\to\infty}\rho_{t}<\infty,$$
\begin{equation*}
0<\liminf_{t\rightarrow \infty }H(\rho_{t})\leq \limsup_{t\rightarrow
\infty }H(\rho_{t})<\infty ,
\end{equation*}%
and 
\begin{equation*}
0<\liminf_{t\rightarrow \infty }\sigma^{2}(\rho_{t})\leq \limsup_{t\rightarrow
\infty }\sigma^{2}(\rho_{t})<\infty.
\end{equation*}
\end{proof}

\begin{proof}[Proof of estimate in (\protect\ref{eq01})]
We carry out this proof in several steps. 
~\newline
\noindent\textbf{Step 1: An useful representation for $h_{x}^{J}$.} Let $%
\lambda\geq 0,$ fixed. We know that%
\begin{eqnarray}
h_{z}^{J}(t) &=&\int_{0}^{z}\mathbb{P}(X_{t}\in dy)\overline{\Pi }(z-y)  \label{a} \\
&=&e^{\lambda z}\int_{0}^{z}e^{-\lambda y}\mathbb{P}(X_{t}\in dy)e^{-\lambda (z-y)}%
\overline{\Pi }(z-y),
\end{eqnarray}
and we note that for $\beta\in\mathbb{C},$ such that $\Re\beta\geq 0,$ 
\begin{equation*}
\int_{0}^{\infty }e^{-\beta y}\overline{\Pi }(y)dy=\frac{\psi_{*}(\beta)}{\beta};
\end{equation*}
when $\beta=0$ the above inequality is understood in the limiting sense.
It follows that%
\begin{equation*}
\int_{0}^{\infty }e^{iyz}e^{-\lambda y}h_{y}^{J}(t)dy=\left(\frac{\psi_{*}(\lambda -iz)}{\lambda -iz}\right)e^{-t\psi (\lambda -iz)}.
\end{equation*}%
For each $t>0$ and $\lambda>0$ define a probability density function by the relation 
\begin{equation*}
g^{\lambda}_{t}(y)=\frac{\lambda}{\psi_{*}(\lambda)}e^{-\lambda
y+t\psi(\lambda)}h_{y}^{J}(t),\qquad y\geq 0.
\end{equation*}
This probability density equals that of the convolution of $%
\mathbb{P}(Y^{\lambda}_{t}\in dy):=e^{-\lambda y+t\psi(\lambda)}\mathbb{P}(X_{t}\in dy)$ and $%
\mathbb{P}(Z_{\lambda}\in dy)=\frac{\lambda}{\psi(\lambda)}e^{-\lambda y}\overline{\Pi%
}(y)dy.$  It follows from the above calculations that for any $z\in \mathbb{R}$ 
\begin{equation*}
\begin{split}
\widehat{g}^{\lambda}_{t}(z)&:=\int_{0}^{\infty
}dye^{iyz}g^{\lambda}_{t}(y)\\
&=\exp\{-t\left(\psi (\lambda -iz)-\psi
(\lambda)\right)\}\frac{\psi_{*}(\lambda-iz)}{\lambda-iz}\frac{\lambda}{%
\psi_{*}(\lambda)}.
\end{split}
\end{equation*}
As a consequence of the hypothesis (H) it is proved in the Lemma~\ref{lemma02} that we always have $\widehat{g}^{\lambda}\in L_{1}$. Then by the inversion theorem for Fourier transforms we get the key expression 
\begin{equation}
g_{t}^{\lambda }(y)=\frac{1}{2\pi }\int_{-\infty }^{\infty }e^{-izy}\widehat{%
g}_{t}^{\lambda }(z)dz,\qquad y\in \mathbb{R}.  \label{keyexp}
\end{equation}%
Now take $y=x,$ and $\lambda =\rho_{t}=\rho (x/t),$ and denote $\mu
_{t}=\mathbb{E}(Y_{t})=x.$ Recalling that $\rho_{t}x=t\rho_{t}\psi ^{\prime }(\rho
_{t}),$ we rewrite the above formula as 
\begin{equation}
\begin{split}
\frac{\rho_{t}}{\psi (\rho_{t})}e^{tH(\rho_{t})}h_{x}^{J}(t)&
=g_{t}^{\rho_{t}}(x)=\frac{1}{2\pi }\int_{-\infty }^{\infty }e^{-izx}%
\widehat{g}_{t}^{\rho_{t}}(z)dz \\
& =\frac{1}{2\pi }\int_{-\infty }^{\infty }\mathbb{E}(\exp \left\{ iz\left( \left(
Y_{t}-\mu_{t}\right) +Z_{\rho_{t}}\right) \right\} )dz.
\end{split}%
\end{equation}%
~\newline
\noindent \textbf{Step 2: Taking $s_{t}:=\sqrt{t}\sigma (\rho_{t})$ we
prove the uniform convergence 
\begin{equation}
\frac{s_{t}}{2\pi }\int_{-\infty }^{\infty }\mathbb{E}(\exp \left\{ iz\left( \left(
Y_{t}-\mu_{t}\right) +Z_{\rho_{t}}\right) \right\} )dz\xrightarrow[]{}%
\frac{1}{\sqrt{2\pi }}  \label{b}
\end{equation}%
} For that end, we first notice that by a change of variables (\ref{b})
becomes 
\begin{equation}\label{53}
\frac{1}{2\pi }\int_{-\infty }^{\infty }\mathbb{E}\left( \exp \left\{ iz\left( \frac{%
Y_{t}-\mu_{t}}{s_{t}}\right) \right\} \right) \mathbb{E}\left( \exp \left\{ iz\frac{%
Z_{\rho_{t}}}{s_{t}}\right\} \right) dz\xrightarrow[]{}\frac{1}{\sqrt{2\pi }%
}
\end{equation}%
~\newline
\noindent \textbf{Step 2.1: We show that $\mathbb{E}\left(Z_{\rho
_{t}}/s_{t}\right)\rightarrow 0$ uniformly.} We write 
\begin{equation*}
\begin{split}
\mathbb{E}(Z_{\rho_{t}})& =\frac{\rho_{t}}{\psi_{*}(\rho_{t})}
\int_{0}^{\infty }ye^{-\rho_{t}y}\overline{\Pi }(y)dy=\frac{\rho_{t}}{\psi_{*}
(\rho_{t})}\int_{(0,\infty )}\Pi (dz)\int_{0}^{z}ye^{-\rho_{t}y}dy \\
& =\frac{1}{\rho_{t}\left(\psi_{*}(\rho_{t})\right)}\left( \psi_{*}(\rho_{t})-\rho_{t}\psi
^{\prime }_{*}(\rho_{t})\right) \leq \frac{1}{\rho_{t}},
\end{split}%
\end{equation*}%
and the result follows since we know $\rho_{t}s_{t}\rightarrow \infty $
uniformly.

\noindent \textbf{Step 2.2.} We now split the integral in (\ref{53}) in four terms 
\begin{equation}
\begin{split}
& \int_{-\infty }^{\infty }\mathbb{E}\left( \exp \left\{ iz\left( \frac{Y_{t}-\mu_{t}%
}{s_{t}}\right) \right\} \right) \mathbb{E}\left( \exp \left\{ iz\frac{Z_{\rho_{t}}}{%
s_{t}}\right\} \right) dz \\
& =\int_{|z|>1/4\Lambda }\mathbb{E}\left( \exp \left\{ iz\left( \frac{Y_{t}-\mu_{t}}{%
s_{t}}\right) \right\} \right) \mathbb{E}\left( \exp \left\{ iz\frac{Z_{\rho_{t}}}{%
s_{t}}\right\} \right) dz+\int_{|z|\leq 1/4\Lambda }e^{-z^{2}/2}dz \\
& +\int_{|z|\leq 1/4\Lambda }\left[ \mathbb{E}\left( \exp \left\{ iz\left( \frac{%
Y_{t}-\mu_{t}}{s_{t}}\right) \right\} \right) -e^{-z^{2}/2}\right] \mathbb{E}\left(
\exp \left\{ iz\frac{Z_{\rho_{t}}}{s_{t}}\right\} \right) dz \\
& -\int_{|z|\leq 1/4\Lambda }e^{-z^{2}/2}\left[ 1-\mathbb{E}\left( \exp \left\{ iz%
\frac{Z_{\rho_{t}}}{s_{t}}\right\} \right) \right] dz \\
& =I+II+III+IV
\end{split}
\label{splitI-IV}
\end{equation}
By Lemma~\ref{lemma:06} we know that $\Lambda \rightarrow 0$ uniformly in $%
x. $ It follows that $II\rightarrow \sqrt{2\pi }.$ Also by the
inequality in Lemma~\ref{lemma:06} it is straightforward that $%
III\rightarrow 0.$ We can bound $IV$ in modulus by%
\begin{equation*}
\frac{\mathbb{E}\left(Z_{\rho_{t}}\right)}{s_{t}}\int_{|z|\leq 1/4\Lambda }|z|e^{-z^{2}/2}dz,
\end{equation*}%
so by the previous result it remains only to verify that 
\begin{equation}
\left\vert \int_{|z|>1/4\Lambda }\mathbb{E}\left( \exp \left\{ iz\left( \frac{%
Y_{t}-\mu_{t}}{s_{t}}\right) \right\} \right) \mathbb{E}\left( \exp \left\{ iz\frac{%
Z_{\rho_{t}}}{s_{t}}\right\} \right) dz\right\vert \rightarrow 0,  \label{e}
\end{equation}%
~\newline
\noindent \textbf{Step 2.3: In the setting ($SC_{\infty }$) or $(G),$ and (H) is
satisfied, the estimate (\ref{e}) holds.} From the Lemma~\ref{lemma:06} and
Lemma~\ref{lemma2} we know that there is a constant $k_{1}$ such that $%
s_{t}\Lambda \leq k_{1}/\rho_{t}.$ Using this fact and making elementary
manipulations we deduce the following upper bound 
\begin{equation}
\begin{split}
& \left\vert \int_{|z|>1/4\Lambda }\mathbb{E}\left( \exp \left\{ iz\left( \frac{%
Y_{t}-\mu_{t}}{s_{t}}\right) \right\} \right) \mathbb{E}\left( \exp \left\{ iz\frac{%
Z_{\rho_{t}}}{s_{t}}\right\} \right) dz\right\vert \\
& \leq s_{t}\int_{|z|>1/4s_{t}\Lambda }\exp \left\{ -t\Re \left\{ \psi
_{\rho_{t}}\left( -iz\right) \right\} \right\} \left\vert \frac{\psi_{*}(\rho
_{t}-iz)}{\rho_{t}-iz}\frac{\rho_{t}}{\psi_{*}(\rho_{t})}\right\vert dz \\
& \leq s_{t}\int_{k_{2}\rho_{t}<|z|<1}\exp \left\{ -t\Re \left\{ \psi
_{\rho_{t}}\left( -iz\right) \right\} \right\} \left\vert \frac{\psi_{*}(\rho
_{t}-iz)}{\rho_{t}-iz}\frac{\rho_{t}}{\psi_{*}(\rho_{t})}\right\vert dz \\
& \quad +s_{t}\int_{|z|>1}\exp \left\{ -t\Re \left\{ \psi_{\rho_{t}}\left(
-iz\right) \right\} \right\} \left\vert \frac{\psi_{*}(\rho
_{t}-iz)}{\rho_{t}-iz}\frac{\rho_{t}}{\psi_{*}(\rho_{t})}\right\vert dz \\
& =:A+B
\end{split}
\label{eq:43}
\end{equation}%
~\newline
\noindent \textbf{Step 2.3.1: $(SC_{\infty})$ case.} In order to prove that $A\rightarrow 0$
uniformly, we start by observing that the hypothesis of stochastic
compactness at infinity $(SC_{\infty }),$ and Proposition 2.2.1 in \cite{BGT}
imply that for any $\alpha_{0}\in (2-\alpha ,2)$ there are constants $k_{4}$
and ${k_{5}}$ such that 
\begin{equation*}
\frac{\int_{0}^{u}z\overline{\Pi }(z)dz}{\int_{0}^{v}z\overline{\Pi }(z)dz}%
\leq k_{4}\left( \frac{u}{v}\right) ^{\alpha_{0}},\qquad u\geq v\geq {k_{5}}
\end{equation*}%
and thus 
\begin{equation}
\frac{Q_{\Pi }(u)}{Q_{\Pi }(v)}\leq k_{4}\left( \frac{u}{v}\right) ^{\alpha
_{0}-2},\qquad u\geq v\geq {k_{5}}.  \label{42a}
\end{equation}%
We fix $\alpha_{0}\in (2-\alpha ,2),$ take $\overline{\rho }%
>\sup_{\{t>1\}}\rho_{t},$ and choose $v_{0}>1$ such that ${k_{5}}\overline{%
\rho }\vee \left( \frac{1}{k_{3}}\right) <v_{0}.$ Next, we bound $A$ 
above as follows 
\begin{equation}
\begin{split}
A& =s_{t}\int_{k_{2}\rho_{t}<|z|<1}\exp \left\{ -t\Re \left\{ \psi_{\rho
_{t}}\left( -iz\right) \right\} \right\} \left\vert \frac{\psi_{*}(\rho_{t}-iz)
}{\rho_{t}-iz}\frac{\rho_{t}}{\psi_{*}(\rho_{t})}\right\vert dz \\
& \leq s_{t}\rho_{t}\int_{k_{2}<|\theta |<1/\rho_{t}}\exp \left\{ -t\Re
\left\{ \psi_{\rho_{t}}\left( -i\theta \rho_{t}\right) \right\} \right\}
d\theta =:A_{1}.
\end{split}%
\end{equation}%
To describe the behaviour of $A_{1}$ we start by lower bounding the exponent
of the integrand as follows. For $\theta \in ((v_{0})^{-1},(k_{5}\rho
_{t})^{-1}),$ or equivalently $k_{5}<(\theta \rho_{t})^{-1}<v_{0}/\rho_{t}$
\begin{equation*}
\begin{split}
\Re \psi_{\rho_{t}}(-i\theta \rho_{t})& =\int_{0}^{\infty }(1-\cos (\theta
\rho_{t}y))e^{-\rho_{t}y}\Pi (dy) \\
& \geq k_{6}e^{-1/\theta }K_{\Pi }(1/(\theta \rho_{t})) \\
& \geq k_{6}e^{-v_{0}}\inf_{u\geq 1/(v_{0}\rho_{t})}\left\{ \frac{K_{\Pi
}(u)}{Q_{\Pi }(u)}\right\} \frac{Q_{\Pi }(1/(\theta \rho_{t}))}{Q_{\Pi
}(v_{0}/\rho_{t})}Q_{\Pi }(v_{0}/\rho_{t}) \\
& \geq k_{7}\inf \left\{ \frac{K_{\Pi }(u)}{Q_{\Pi }(u)},{u\geq \left( v_{0}%
\overline{\rho }\right) ^{-1}}\right\} (\theta v_{0})^{2-\alpha_{0}}Q_{\Pi
}(v_{0}/\rho_{t}),
\end{split}%
\end{equation*}%
where in the last inequality we used (\ref{42a}). The later together with
the inequality 
\begin{equation*}
v_{0}^{2}Q_{\Pi }(v_{0}/\rho_{t})=\rho_{t}^{2}\int_{0}^{v_{0}/\rho_{t}}y%
\overline{\Pi }(y)dy\geq \rho_{t}^{2}\int_{0}^{1/\rho_{t}}y\overline{\Pi }%
(y)dy=Q_{\Pi }(1/\rho_{t})\geq H(\rho_{t}),
\end{equation*}%
imply 
\begin{equation*}
t\Re \psi_{\rho_{t}}(-i\theta \rho_{t})\geq k_{8}\theta ^{2-\alpha
_{0}}(v_{0})^{-\alpha_{0}}tH(\rho_{t}),
\end{equation*}%
for $t$ large enough, uniformly in $x.$ Applying this in $A_{1}$ and the
results from Lemma~\ref{lemma2} we obtain 
\begin{equation*}
\begin{split}
A_{1}& \leq \sqrt{t}\sigma (\rho_{t})\rho_{t}\int_{1/v_{0}}^{1/\rho_{t}{%
k_{5}}}\exp \{-k_{8}\theta ^{2-\alpha_{0}}(v_{0})^{-\alpha_{0}}tH(\rho
_{t})\}d\theta \\
& \leq \sqrt{2tH(\rho_{t})}\int_{1/v_{0}}^{1/\rho_{t}k_{5}}\exp
\{-k_{9}\theta ^{2-\alpha_{0}}tH(\rho_{t})\}d\theta ,
\end{split}%
\end{equation*}%
where $k_{9}=k_{8}v_{0}^{-\alpha }.$ Recall that $tH(\rho_{t})\rightarrow
\infty ,$ so that putting $\theta ^{2-\alpha_{0}}tH(\rho_{t})=z^{2-\alpha
_{0}}$ gives 
\begin{equation*}
A_{1}\leq (\sqrt{2}tH(\rho_{t}))^{\frac{-\alpha_{0}}{2(2-\alpha_{0})}%
}\int_{(tH(\rho_{t})/v_{0})^{\frac{1}{2-\alpha_{0}}}}^{\infty }\exp
\{-k_{9}z^{2-\alpha_{0}}\}dz\rightarrow 0.
\end{equation*}%
~\newline
\noindent \textbf{Step 2.3.2: $(SC_{\infty})$ case.} We next prove that $B\rightarrow 0.$
Proceeding as above we easily get that for $\theta >1/\rho_{t}{k_{5}}$ 
\begin{equation*}
\Re \psi_{\rho_{t}}(-i\theta \rho_{t})\geq k_{10}e^{-v_{0}}Q_{\Pi
}(1/\theta \rho_{t})\geq k_{10}e^{-v_{0}}Q_{\Pi }(k_{5}):=k_{11}.
\end{equation*}%
Now, we apply this estimate to $B$ to get that for $t>t_{0}$ 
\begin{equation*}
B\leq s_{t}e^{-(t-\widetilde{t}_{0})k_{11}}\frac{\rho_{t}}{\psi_{*}(\rho_{t})}%
\int_{|z|>1}\exp \left\{ -\widetilde{t}_{0}\Re \left\{ \psi_{\rho
_{t}}\left( -iz\right) \right\} \right\} \left\vert \frac{\psi_{*} (\rho_{t}-iz)%
}{\rho_{t}-iz}\right\vert dz.
\end{equation*}%
Observe that by Lemma~\ref{lemma02} the latter integral, as a function of $\rho_{t},$ is uniformly bounded.  This will be enough to conclude the argument because we already know that $%
\frac{\rho_{t}}{\psi_{*}(\rho_{t})}\rightarrow 1/(\mathbb{E}(X_{1})-b)<\infty $ and it is
easy to verify that $s_{t}$ grows at most as a power function of $t$. The latter
is actually true because Lemma~\ref{lemma2} allows us to ensure that 
\begin{equation*}
s_{t}\leq \frac{1}{\rho_{t}}\sqrt{tH(\rho_{t})}\leq \frac{1}{\rho_{t}}%
\sqrt{tH(\overline{\rho })},
\end{equation*}%
 with $\overline{\rho }=\sup_{t}\rho_{t}<\infty ;$ and moreover given that 
$tH(\rho_{t})\rightarrow \infty $ we deduce from (\ref{42a}) that for all
large enough $t$ 
\begin{equation*}
t\rho_{t}^{2-\alpha_{0}}\geq k_{13}tQ_{\Pi }(1/\rho_{t})\geq
k_{13}tH(\rho_{t})\geq k_{13}>0,
\end{equation*}%
that is 
\begin{equation*}
\rho_{t}\geq k_{14}t^{-1/(2-\alpha_{0})},\qquad \text{for }t\ \text{large
enough}.
\end{equation*}%
Which implies the claimed fact. ~\newline
\noindent \textbf{Step 2.3.3: case $(G).$} 
We will prove that the term
\begin{equation*}
s_{t}\int_{|z|>1/4s_{t}\Lambda }\exp \left\{ -t\Re \left\{ \psi
_{\rho_{t}}\left( -iz\right) \right\} \right\} \left\vert \frac{\psi_{*}(\rho
_{t}-iz)}{\rho_{t}-iz}\frac{\rho_{t}}{\psi_{*}(\rho_{t})}\right\vert dz\,
\end{equation*}
tends to $0$ uniformly in $x$. Recall that in this setting we have $$0<\underline{\rho}:=\liminf_{t\to\infty}\rho_{t}\leq \limsup_{t\to\infty}\rho_{t}:=\overline{\rho}<\infty.$$ Using this, the definition of $s_{t}\Lambda,$ and the calculations used in the proof of (\ref{lb4}) it is easy to check that $s_{t}\Lambda$ is bounded by below by a strictly
positive constant, say $l^{\ast }.$  Also, as we required $X$ to be 
strongly non-lattice, and this is a property that is preserved under change of
measure, we have that 
\begin{equation}
\begin{split}
\liminf_{\theta \rightarrow \infty }{\Re }(\psi_{\rho_{t}}(-i\theta ))&
=\liminf_{\theta \rightarrow \infty }\int_{0}^{\infty }(1-\cos (\theta
y))e^{-\rho_{t}y}\Pi (dy) \\
& \geq \liminf_{\theta \rightarrow \infty }\int_{0}^{\infty }(1-\cos (\theta
y))e^{-\overline{\rho }y}\Pi (dy)>0.
\end{split}%
\end{equation}%
We denote $\widetilde{\psi }_{\overline{\rho }}(\theta )=\int_{0}^{\infty
}(1-\cos (\theta y))e^{-\overline{\rho }y}\Pi (dy),$ and $m(s)=\inf_{\theta
\geq s}\widetilde{\psi }_{\overline{\rho }}(\theta ).$ The above
observations and the continuity of $\widetilde{\psi }_{\overline{\rho }%
}(\theta )$ imply that $m(s)>0,$ for all $s>0.$ It follows that for $t>t_{0}$
\begin{equation*}
\begin{split}
& s_{t}\int_{l}^{\infty }e^{-t{\Re }(\psi_{\rho
_{t}}(-i\theta ))}\left\vert \frac{\psi_{*}(\rho
_{t}-i\theta)}{\rho_{t}-i\theta}\frac{\rho_{t}}{\psi_{*}(\rho_{t})}\right\vert d\theta \\
& \leq \sqrt{t}\sigma (\rho_{t})e^{-(t-t_{0})m(l^{\ast })}\int_{l^{\ast
}}^{\infty }e^{-t_{0}\widetilde{\psi }_{\overline{\rho }}(\theta )}\left\vert \frac{\psi_{*}(\rho
_{t}-i\theta)}{\rho_{t}-i\theta}\frac{\rho_{t}}{\psi_{*}(\rho_{t})}\right\vert 
d\theta.
\end{split}%
\end{equation*}%
By Lemma~\ref{lemma02} the right most term tends to $0$ uniformly in $x$. 
\end{proof}

The proofs of the estimates (\ref{eq03}) and (\ref{eq04}) use arguments very
similar to those used in the previous proof, and hence in the forthcoming
lines we will only outline the keys facts needed to adapt that proof. 
\begin{proof}[Proof of the estimate (\protect\ref{eq03})]
We proceed as before, using Lemma~\ref{CJF} we define a probability density 
\begin{equation}
\begin{split}
& \mathcal{Q}_{t}^{\lambda }(y)=\frac{\lambda \psi (\lambda )e^{t\psi
(\lambda )}}{\psi_{*}(\lambda)(1-e^{-\Delta \psi (\lambda )})}%
e^{-\lambda y}h_{y}^{J}(t,\Delta ) \\
& =\frac{\lambda \psi (\lambda )e^{t\psi (\lambda )}}{\psi_{*}(\lambda)(1-e^{-\Delta \psi (\lambda )})}e^{-\lambda y}\mathbb{P}(T_{y}\in
(t,t+\Delta ],X_{T_{y}}>y) \\
& =\int_{0}^{y}\mathbb{P}(X_{t}\in da)e^{-\lambda a}e^{t\psi (\lambda
)}\frac{\psi (\lambda )}{(1-e^{-\Delta \psi (\lambda )})}%
\int_{0}^{y-a}U_{\Delta}(dz)e^{-\lambda z} \\
& \qquad \times \frac{\lambda }{\psi_{*}(\lambda)}e^{-\lambda
(y-a-z)}\overline{\Pi }(y-a-z).
\end{split}%
\end{equation}%
We easily verify from the above expression that this is the density of the
sum of the three independent random variables, $Y_{t}^{\lambda },$ $%
Z_{\lambda },$ and $W_{\lambda },$ with $Y_{t}^{\lambda },$ and $Z_{\lambda
},$ as defined in the proof of estimate~(\ref{eq01}), and $W_{\lambda }$
that follows the probability law 
\begin{equation*}
\mathbb{P}(W_{\lambda }\in dy)=\frac{\psi (\lambda )}{%
(1-e^{-\Delta \psi (\lambda )})}U_{\Delta}(dy)e^{-\lambda y}.
\end{equation*}%
We can therefore proceed as in the proof of estimate (\ref{eq01}) replacing $%
Z_{\rho_{t}}$ by $Z_{\rho_{t}}+W_{\rho_{t}}.$ But for that end we should
first prove that the Fourier transform of $\mathcal{Q}_{t}^{\lambda }$ is
integrable. This is a straightforward consequence of the fact that 
\begin{equation*}
\left\vert \mathbb{E}(\exp \{i\beta (Y_{t}^{\lambda }+Z_{\lambda }+W_{\lambda
})\})\right\vert \leq \left\vert \mathbb{E}(\exp \{i\beta (Y_{t}^{\lambda
}+Z_{\lambda })\})\right\vert
\end{equation*}%
and that we already proved that the rightmost term in the above inequality
is integrable. We should now prove that $\mathbb{E}\left(W_{\rho_{t}}/s_{t}\right)$ tend to zero
, uniformly in $x$ and in $\Delta.$ We have 
\begin{equation}
\left\vert \mathbb{E}\left( \exp \left\{ i\beta \frac{W_{\rho_{t}}}{s_{t}}\right\}
\right) -1\right\vert \leq \mathbb{E}\left(\frac{|\beta |W_{\rho_{t}}}{s_{t}}%
\right) ,
\end{equation}%
and 
\begin{equation}
\begin{split}
\mathbb{E}\left(\frac{W_{\rho_{t}}}{s_{t}}\right) & =\frac{1}{s_{t}}\frac{\psi
(\rho_{t})}{1-e^{-\Delta \psi (\rho_{t})}}\int_{0}^{\Delta }ds\mathbb{E}\left(
X_{s}e^{-\rho_{t}X_{s}}\right) \\
& =\frac{1}{s_{t}}\frac{\psi (\rho_{t})}{1-e^{-\Delta \psi (\rho_{t})}}%
\int_{0}^{\Delta }dse^{-s\psi (\rho_{t})}\mathbb{E}\left( X_{s}e^{-\rho
_{t}X_{s}+s\psi (\rho_{t})}\right) \\
& =\frac{1}{s_{t}}\frac{\psi (\rho_{t})}{1-e^{-\Delta \psi (\rho_{t})}}%
\int_{0}^{\Delta }dse^{-s\psi (\rho_{t})}\mathbb{E}\left( Y_{s}\right) \\
& =\frac{x_{t}}{s_{t}}\frac{\psi (\rho_{t})}{1-e^{-\Delta \psi (\rho_{t})}}%
\int_{0}^{\Delta }dse^{-s\psi (\rho_{t})}s \\
& =\frac{x_{t}}{s_{t}}\frac{\psi (\rho_{t})}{1-e^{-\Delta \psi (\rho_{t})}}%
\frac{1}{\left( \psi (\rho_{t})\right) ^{2}}\left( 1-e^{-\Delta \psi (\rho
_{t})}-\Delta \psi (\rho_{t})e^{-\Delta \psi (\rho_{t})}\right) \\
& \leq \frac{\rho_{t}x_{t}}{\psi (\rho_{t})}\frac{1}{\sqrt{t\rho
_{t}^{2}\sigma ^{2}(\rho_{t})}}.
\end{split}%
\end{equation}%
The rightmost term in the above equation converges to zero uniformly in $x$
and $\Delta $ because 
\begin{equation*}
\frac{\rho_{t}x_{t}}{\psi (\rho_{t})}=\frac{\rho_{t}\psi ^{\prime }(\rho
_{t})}{\psi (\rho_{t})}\leq 1,
\end{equation*}%
which is in turn an easy consequence of the elementary inequality 
\begin{equation}
\begin{split}
\psi ^{\prime }(\lambda )& =b+\int_{0}^{\infty }ye^{-\lambda y}\Pi
(dy)=b+\int_{0}^{\infty }da\int_{a}^{\infty }e^{-\lambda y}\Pi (dy) \\
& \leq b+\int_{0}^{\infty }dae^{-\lambda a}\overline{\Pi }(a)= \frac{\psi
(\lambda )}{\lambda },
\end{split}%
\end{equation}%
for all $\lambda >0.$
\end{proof}

\begin{proof}[Proof of the estimate (\protect\ref{eq04})]
By Lemma \ref{CJF} we have the key identity 
\begin{equation*}
h_{y}^{C}(t,\Delta )=\mathbb{P}(T_{y}\in (t,t+\Delta
],X_{T_{y}}=y)=b\int_{[0,y]}\mathbb{P}(X_{t}\in dz)u_{\Delta }(y-z).
\end{equation*}%
Taking Laplace transform in $y$ we obtain 
\begin{equation*}
\int_{0}^{\infty }dye^{-\lambda y}h_{y}^{C}(t,\Delta )=be^{-t\psi (\lambda )}%
\frac{(1-e^{-\Delta \psi (\lambda )})}{\psi (\lambda )}.
\end{equation*}%
for any $t>0.$ Observe the identity 
\begin{equation*}
\mathbb{P}(W_{\lambda }\in dy)=\frac{\psi (\lambda )}{(1-e^{-\Delta \psi (\lambda )})}%
e^{-\lambda y}u_{\Delta }(y)dy,
\end{equation*}%
with $W_{\lambda }$ as defined in the previous proof. We deduce therefrom
the identity 
\begin{equation}\label{AF}
b\mathbb{P}(Y_{t}^{\lambda }+W_{\lambda }\in dy)=e^{t\psi (\lambda )}\frac{\psi
(\lambda )}{(1-e^{-\Delta \psi (\lambda )})}e^{-\lambda y}h_{y}^{C}(t,\Delta
)dy,\qquad y\geq 0.
\end{equation}%
The Fourier transform of the left most term in the
above equation is integrable because of the inequality $$\left|\frac{1-e^{-\Delta\psi(\lambda-i\theta)}}{\psi(\lambda-i\theta)}\right|\leq \left|\frac{1}{\psi(\lambda-i\theta)}\right|\sim\left|\frac{1}{b\theta}\right|,$$ the hypothesis (H), the Lemma~\ref{lemma02} and Proposition~2 in Chapter 1 in \cite{bertoinbook}.  We then deduce the identity 
\begin{equation}
\begin{split}
& s_{t}\frac{\psi (\rho_{t})}{(1-e^{-\Delta \psi (\rho_{t})})}e^{tH(\rho
_{t})}h_{x}^{C}(t,\Delta ) \\
& =\frac{b}{2\pi }\int_{-\infty }^{\infty }dz\mathbb{E}\left( \exp \left\{ iz\left(
\left( \frac{Y_{t}-\mu_{t}}{s_{t}}\right) +\frac{W_{\rho_{t}}}{s_{t}}%
\right) \right\} \right) .
\end{split}%
\end{equation}%
Using the arguments in the previous proofs we get that the rightmost term in
the above identity equals 
\begin{equation*}
\frac{b}{\sqrt{2\pi }}(1+o(1)),
\end{equation*}%
and the error term is uniform in $x$ and $\Delta .$ Which finishes the proof.
\end{proof}

\section{Proof of Proposition~\protect\ref{prop:rvcase}}

We repeat the calculation on page 9 with $\lambda =0$ to get%
\begin{equation*}
\hat{h}_{z}^{J}(t):=\int_{0}^{\infty }e^{izy}h_{y}^{J}(t)dy=e^{-t\psi (-iz)}%
\frac{\psi _{\ast }(-iz)}{-iz},
\end{equation*}%
so that 
\begin{equation}  \label{rho0}
th_{x}^{J}(t)=\frac{t}{2\pi }\int_{-\infty }^{\infty }e^{-ixz}e^{-t\psi
(-iz)}\frac{\psi _{\ast }(-iz)}{-iz}dz.
\end{equation}
The integral above is well defined since the hypothesis (H) ensures the
integrability in a neighbourhood of infinity, and that around zero follows
from the regular variation of $\overline{\Pi}$ at infinity. Indeed, the
regular variation of $\overline{\Pi}$ implies the finiteness of the integral 
\begin{equation*}
\int^{\infty}_{1}\frac{dz}{z}\overline{\Pi}(z)<\infty,
\end{equation*}
and some elementary calculations allow to deduce therefrom that 
\begin{equation*}
\int_{|z|<1}\left|\frac{\psi _{\ast }(-iz)}{-iz}\right|dz<\infty.
\end{equation*}
We write the RHS of (\ref{rho0}) as $I_{1}+I_{2,\text{ }}$where%
\begin{eqnarray*}
I_{1} &=&\frac{t}{2\pi }\int_{|z|\leq Kc(t)}e^{-ixz}e^{-t\psi (-iz)}\frac{%
\psi _{\ast }(-iz)}{-iz}dz \\
&=&\frac{1}{2\pi }\int_{|z|\leq K}e^{-izy_{t}}e^{-t\psi (-iz/c(t))}\frac{%
t\psi _{\ast }(-iz/c(t))}{-iz}dz \\
&=&\frac{1}{2\pi }\int_{|z|\leq K}e^{-izy_{t}}e^{-\tilde{\psi}(-iz)}\frac{%
\tilde{\psi}(-iz)}{-iz}dz+o(1),
\end{eqnarray*}%
where $\tilde{\psi}$ is the exponent of the limiting stable process $S,$ and
we use the fact that $t\psi _{\ast }(-iz/c(t))\backsim t\psi
(-iz/c(t))\rightarrow \tilde{\psi}(-iz)$ uniformly on $[-K,K]$. Clearly%
\begin{equation*}
\lim_{K\rightarrow \infty }|\int_{|z|>K}e^{-\tilde{\psi}(-iz)}\frac{\tilde{%
\psi}(-iz)}{-iz}dz|=0,
\end{equation*}%
so that%
\begin{equation*}
\lim_{K\rightarrow \infty }\lim_{t\rightarrow \infty }|I_{1}-\tilde{h}%
_{y_{t}}(1)|=0,\text{ uniformly in }y_{t}.
\end{equation*}%
The result follows because, for any fixed $K$%
\begin{equation*}
\begin{split}
&\lim_{t\rightarrow \infty }|I_{2}| \\
&\leq \lim_{t\to\infty} te^{-(t-t_{0})\kappa }\int_{|z|>Kc(t)}\exp \left\{
-t_{0}\int_{0}^{\infty }(1-\cos (zy))\Pi (dy)\right\} \frac{|\psi (-iz)|}{z}%
dz \\
&=0,
\end{split}%
\end{equation*}
where $\kappa =\liminf_{|z|\to\infty} \int_{0}^{\infty }(1-\cos (zy))\Pi
(dy)>0,$ by the strongly non-lattice assumption.

Similarly we have the representation%
\begin{equation*}
c(t)h_{x}^{C}(t,\Delta )=\frac{bc(t)}{2\pi }\int_{-\infty }^{\infty
}e^{-ixz}e^{-t\psi (-iz)}\frac{(1-e^{-\Delta \psi (-iz)})}{\psi (-iz)}dz,
\end{equation*}%
the integrability following from (H) and the bound%
\begin{equation*}
\begin{split}
\left|\frac{(1-e^{-\Delta \psi (-iz)})}{\psi (-iz)}\right| &\leq \left|\frac{%
1}{-ibz+\psi _{\ast }(-iz)}\right| \\
&\backsim \frac{1}{b|z|}\text{ as }|z|\rightarrow \infty .
\end{split}%
\end{equation*}

Again we have the uniform estimate,%
\begin{equation*}
\begin{split}
&\frac{bc(t)}{2\pi }\int_{|z|\leq Kc(t)}e^{-ixz}e^{-t\psi (-iz)}\frac{%
(1-e^{-\Delta \psi (-iz)})}{\psi (-iz)}dz \\
&=\frac{b}{2\pi }\int_{|z|\leq K}e^{-izy_{t}}e^{-t\psi (-iz/c(t))}\frac{%
(1-e^{-\Delta \psi (-iz/c(t))})}{\psi (-iz/c(t))}dz \\
&=\frac{b\Delta }{2\pi }\int_{|z|\leq K}e^{-izy_{t}}e^{-t\tilde{\psi}%
(-iz)}dz+o(1),
\end{split}%
\end{equation*}

and the proof is concluded as before.

\end{document}